\documentclass[11pt]{amsart}
\usepackage[utf8]{inputenc}
\usepackage[margin=1in]{geometry}
\usepackage{amsmath}
\usepackage{amssymb}
\usepackage{amsfonts}
\usepackage{amscd}
\usepackage{color}
\usepackage{comment}
\usepackage{slashed}
\usepackage{wasysym}
\usepackage{pdflscape}
\usepackage{tikz}
\usepackage{tikz-cd}
\usepackage{shadow}
\usepackage{enumerate}
\usepackage{multirow}
\usepackage{physics}
\usepackage{ytableau}
\usepackage{stackengine}
\usepackage{stmaryrd}
\usepackage{MnSymbol}
\usepackage{graphicx} 
\usepackage{bbold}
\usepackage{mathrsfs}
\usepackage{t1enc}

\usepackage{cite}
\usepackage[backref]{hyperref}

\renewcommand*{\backref}[1]{}
\renewcommand*{\backrefalt}[4]{%
	\ifcase #1%
	\or [Page~#2.]%
	\else [Pages~#2.]%
	\fi%
}
\newcommand{\hh}{{\hspace{.3mm}}}

\newcommand{\pdot}{{\boldsymbol{\cdot}}}

\newcommand{\bm}[1]{\mbox{\boldmath $ #1 $}}

\newtheorem{theorem}{Theorem}[section]
\newtheorem{lemma}[theorem]{Lemma}

\newtheorem{proposition}[theorem]{Proposition}

\theoremstyle{definition}

\theoremstyle{remark}
\newtheorem{remark}[theorem]{Remark}

\newtheorem{problem}[theorem]{Problem}

\let\ccdot\cdot
\def\cdot{\hbox to 2.5pt{\hss$\ccdot$\hss}}

\newcommand{\Ad}{\operatorname{Ad}}

\newcommand{\End}{\operatorname{End}}

\usepackage{ifthen}

\def\sideremark#1{\ifvmode\leavevmode\fi\vadjust{\vbox to0pt{\vss% the remark
			\hbox to 0pt{\hskip\hsize\hskip1em%                          will appear only
				\vbox{\hsize3cm\tiny\raggedright\pretolerance10000%          on the side
					\noindent #1\hfill}\hss}\vbox to8pt{\vfil}\vss}}}%
\newcommand\sss{\scriptscriptstyle}

\newcommand{\qi}{{\bm i}}
\newcommand{\qj}{{\bm j}}
\newcommand{\qk}{{\bm k}}
\newcommand{\ZZ}{\mathbb{Z}}
\newcommand{\RR}{\mathbb{R}}
\newcommand{\CC}{\mathbb{C}}
\newcommand{\HH}{\mathbb{H}}

\newcommand{\GL}{\operatorname{GL}}
\newcommand{\Sp}{\operatorname{Sp}}

\newcommand{\Uni}{\operatorname{U}}

\newcommand{\Mp}{\operatorname{Mp}}
\newcommand{\Mpc}{\operatorname{Mp^c}}

\newcommand{\Id}{\operatorname{Id}}
\newcommand{\Span}{\operatorname{Span}}

\newcommand{\Stab}{\operatorname{Stab}}
\newcommand{\frg}{\mathfrak{g}}
\newcommand{\frh}{\mathfrak{h}}
\newcommand{\frm}{\mathfrak{m}}

\newcommand{\uni}{\mathfrak{u}}

\renewcommand{\sp}{\mathfrak{sp}}

\newcommand{\heis}{\mathfrak{heis}}
\newcommand{\gl}{\mathfrak{gl}}

\def\Re{\mathrm{Re}}

\numberwithin{equation}{section}

\makeatletter
\def\@tocline#1#2#3#4#5#6#7{\relax
	\ifnum #1>\c@tocdepth % then omit
	\else
	\par \addpenalty\@secpenalty\addvspace{#2}%
	\begingroup \hyphenpenalty\@M
	\@ifempty{#4}{%
		\@tempdima\csname r@tocindent\number#1\endcsname\relax
	}{%
		\@tempdima#4\relax
	}%
	\parindent\z@ \leftskip#3\relax \advance\leftskip\@tempdima\relax
	\rightskip\@pnumwidth plus4em \parfillskip-\@pnumwidth
	#5\leavevmode\hskip-\@tempdima
	\ifcase #1
	\or\or \hskip 1em \or \hskip 2em \else \hskip 3em \fi%
	#6\nobreak\relax
	\dotfill\hbox to\@pnumwidth{\@tocpagenum{#7}}\par
	\nobreak
	\endgroup
	\fi}
\makeatother

\begin{document}
	\date{\today}
	\title{
	  	Seven Sphere Quantization
	}
	\author{Subhobrata Chatterjee${}^\flat$
	}
	\address{${}^{\flat}$Center for Quantum Mathematics and Physics (QMAP) and Department of Physics and Astronomy, University of California, Davis, CA 95616, USA}
	\email{sbhchatterjee@ucdavis.edu}
	\author{Can G\"ormez${}^\sharp$
	}
	\author{Andrew Waldron${}^\sharp$ 
	}
	\address{${}^{\sharp}$Center for Quantum Mathematics and Physics (QMAP) and Department of Mathematics, University of California, Davis, CA 95616, USA}
	\email{cgormez@ucdavis.edu} \email{wally@math.ucdavis.edu}

	\vspace{10pt}
	
	\renewcommand{\arraystretch}{1}

	\begin{abstract} 
	Co-oriented contact manifolds quite generally describe classical dynamical systems. Quantization is achieved by suitably associating a Schr\"odinger equation to every path in the contact manifold.
	We quantize the standard contact seven sphere by treating it as a homogeneous space of the quaternionic unitary group in order to construct a contact analog of Fedosov's formal connection on symplectic spinor bundles. We show that requiring convergence of the formal connection naturally filters the symplectic spinor bundle and yields an exact flat connection on each corresponding subbundle. A key ingredient is a generalization of the Holstein--Primakoff mechanism to the quaternionic unitary group. The passage from formal to \textit{bona fide} quantization determines unitary irreducible representations of the quaternionic unitary group, whose dimensions tend to infinity as the formal deformation parameter approaches its classical limit. This appearance of finite-dimensional representations is not surprising since the contact seven sphere is closed and physically describes generalized positions, momenta and time variables. 
	
		\vspace{2cm}
		\noindent
		%\begin{keywords}
		{\sf \tiny Keywords: Quantization, contact geometry, deformation quantization, quaternionic homogeneous spaces, Holstein--Primakoff, seven sphere.}
		%\end{keywords}
	\end{abstract}
	
	\maketitle
	
	\pagestyle{myheadings} \markboth{Chatterjee, G\"ormez, \& Waldron}{Seven Sphere Quantization}
	
	\tableofcontents
	\newpage
	
\section{Introduction}
Quantization is a bridge between classical and quantum dynamics. There are various formulations of the quantization problem for a classical dynamical system. Two main approaches are geometric and deformation quantization, both of which rely on underlying geometric data. Geometric quantization typically begins with a symplectic manifold over which one constructs a complex line bundle with a connection~\cite{souriau1966,Kostant1970,souriau1997structure,woodhouse1992,Kirillov_2001}. Deformation quantization, on the other hand, starts with a Poisson structure~\cite{BFFLS,BFFLS1,BFFLS2,dewildelecomte1983,fedosovdeform,kontsevich2003,Berezin_1975}. In both cases, one attempts to promote the commutative algebra of classical observables to a noncommutative algebra of quantum observables. In this article, we focus on a geometric, dynamics-first approach to quantization.

We begin with an \textit{odd symplectic manifold} $(Z,\omega)$, where $\omega$ is a maximal rank closed two-form~\cite{Corradini:2020vqa} (see also \cite{He,Lin2013}). Note that every strict contact manifold (as defined below) is odd symplectic. Odd symplectic manifolds are a particularly appealing~\cite{Vaisman1983,Gotay_Śniatycki_1981,Günther_1980} classical starting point because they put generalized momenta, positions and time on the same footing, in much the same way as is done for space and time in general relativity. One can now directly quantize classical dynamics by constructing a connection encoding a generally covariant Schr\"odinger equation~\cite{Herczeg:2017xxy,Herczeg:2018hup}, or in other words, quantum evolutions for all possible classical clocks. For general geometries, quantization remains a difficult problem. We shall therefore focus on the highly symmetric case of a seven sphere equipped with its standard contact structure and believe that our results ought generalize to contact homogeneous spaces~\cite{alekseevsky,alekseevskysymmetric}.

An odd dimensional manifold $Z^{2n+1}$ equipped with a {\it contact form} $\alpha$, meaning a one-form that defines a nowhere vanishing volume form $\alpha\wedge (d \alpha)^{\wedge n}$, is precisely tailored to describe both classical and quantum physics. Firstly, the action functional of paths $\gamma$ in~$Z$, 
$$
S[\gamma]=\int_\gamma \alpha \, ,
$$
is extremized by compactly supported variations of $\gamma$ when
$$
\dot \gamma \in \ker d\alpha \, .
$$
Note that, by virtue of the volume non-degeneracy condition,  the kernel of the {\it Levi two-form} $d\alpha$, viewed as a map $\Gamma(TZ) \to \Gamma(T^*Z)$, defines a line subbundle of $TZ$. The section $R$ thereof, determined by the normalization $\alpha(R)=1$, is termed the {\it Reeb vector field}. The data $(Z^{2n+1},\alpha)$ is called a {\it strict contact manifold}~\cite{Geiges2008}, and provides a time covariant generalization of Hamiltonian dynamics on symplectic manifolds~\cite{LibermannMarle1987,CiagliaCruzMarmo2018,Herczeg:2017xxy,Bravetti2017,Bravetti2017v2,leonlainz2019,leonlainz2021}. 

Secondly, in addition to classical dynamics, the contact form $\alpha$ canonically determines a bundle of Hilbert spaces 
over $Z$~\cite{chwdynamics} as follows: The kernel of $\alpha$, viewed as a map~$\Gamma(TZ)\to C^\infty(Z)$, yields a maximally non-integrable hyperplane distribution~$\xi$ in~$TZ$. The data $(Z,\xi)$ is then a \textit{contact manifold}.
The distribution~$\xi$ is a symplectic vector bundle over $Z$ with bilinear form determined by the Levi two-form. In turn, the bundle of symplectic frames on~$\xi$ is a principal~$\Sp(2n,{\mathbb R})$-bundle. In quantum mechanics, one is interested in the unitary representation of the metaplectic group $\operatorname{Mp}(2n,\RR)$, typically realized by acting on the Hilbert space $L^2({\mathbb R}^n)$ with its standard hermitean inner product. Locally, the principal $\Sp(2n,\RR)$-bundle of symplectic frames lifts to an $\Mp(2n,\RR)$-bundle. In general, such lifts do not exist globally. However, for a symplectic vector bundle, we may always lift its symplectic frame bundle to a principal~$\Mpc(2n,\RR)$-bundle, where the structure group is the metaplectic-c group~$\Mpc(2n,\RR)$~\cite{robinson1989metaplectic}. Hence, we consider the associated vector bundle ${\mathcal H} Z$ determined by a unitary representation of $\Mpc(2n,\RR)$. The bundle ${\mathcal H}Z$ is called the symplectic spinor bundle and
inherits a hermitean bundle inner product $\langle\pdot,\pdot\rangle$
from that on ${\mathcal H}$~\cite{symplecticdirac,symplecticdirac2}; see Section~\ref{sec:quantization} for details.
We will also denote by ${\mathcal H} Z$ the~$\Uni({\mathcal H})$ vector bundle induced by considering orthonormal frames for the inner product $\langle\pdot,\pdot\rangle$ and then again considering the associated vector bundle for the $\Uni({\mathcal H})$ representation~${\mathcal H}$.
\bigskip

The quantization problem for a strict contact manifold $(Z,\alpha)$ is stated, in its strongest form, as follows. 
\begin{problem}\label{wehaveproblems}
Find a one-parameter family of flat connections $\nabla^{(\hbar)}$ on a $\Uni({\mathcal H})$ vector bundle~${\mathcal H}Z$ such that, for all
$\Phi,\Psi\in \Gamma({\mathcal H}Z) $ and  $u \in \Gamma(TZ)$,
\begin{enumerate}[(i)]
\item $ \mathcal{L}_u \langle \Phi, \Psi \rangle = \langle\nabla^{(\hbar)}_u \Phi , \Psi\rangle+\langle\Phi ,\nabla^{(\hbar)}_u \Psi\rangle \, ,$ \\[-2mm]
\item\label{limitformalflat} $\lim_{\hbar \to 0} i \hbar \nabla_u^{(\hbar)} \Psi =  \alpha(u) \Psi \, .$\hfill\scalebox{1.5}{${\large\diamond}$}
\end{enumerate}
\end{problem}
\noindent
To define the limit in Part~\eqref{limitformalflat}, we need a suitable norm. Since physical probabilities are computed using a pointwise Born rule (see Equation~\eqref{bornrule} below), for a section $\Psi \in \Gamma(\mathcal{H}Z)$, at each point $z \in Z$ we define 
\begin{align*}
	\|\Psi\|_z := \sqrt{\langle \Psi(z), \Psi(z) \rangle } \, .
\end{align*}
The limit in Part~\eqref{limitformalflat} is defined at each $z \in Z$ with respect to the above norm.
%\vspace{5 cm} and there are many possible---for our purposes inessential---choices. Since we are interested in cases, such as $S^3$ and~$S^7$, where $Z$ is compact, one natural choice is an $L^2$-norm. This is defined using the volume form on $(Z,\alpha)$ and the hermitean bundle product on $\mathcal{H}Z$. 
This choice will be used tacitly to define limits at several junctures below.

At a fixed value of $\hbar$, the data $({\mathcal H}Z,\nabla^{(\hbar)})$ is an example of a {\it quantum dynamical system}: Given a Hilbert space ${\mathcal H}$ and a manifold $Z$, 
the latter is defined as the data~$(\mathscr{H}Z,\nabla)$ consisting of a $\Uni({\mathcal H})$ hermitean vector bundle $\mathscr{H}Z$ over $Z$ with fiber $\mathcal{H}$ and a flat connection~$\nabla$ preserving the hermitean bundle inner product. We call a connection obeying these properties a \textit{quantum connection}. This connection governs quantum dynamics by parallel transport $$\nabla \Psi = 0 \, ,$$
where $\Psi \in \Gamma(\mathscr{H}Z)$. This is a generalization of the Schr\"odinger equation~\cite{Herczeg:2017xxy,Herczeg:2018hup}. In particular, if $\dot{\gamma}$ is the tangent vector to any path $\gamma$ in $Z$ for some notion of evolution, then $\nabla_{\dot\gamma}\Psi$ is the corresponding Schr\"odinger equation. This allows an implementation of the Born rule as follows: Given an initial state vector $\Psi(z_i)\in \mathscr{H}_{z_i} Z$ for $z_i \in Z$, we can compute the probability $P_{f,i}$ of observing a desired final state vector $\Psi(z_f) \in \mathscr{H}_{z_f} Z$ for $z_f \in Z$ by the Born rule 
\begin{align}\label{bornrule}
	P_{f,i} = \frac{\left| \langle \Psi(z_f) , \mathcal{P}_\gamma \Psi(z_i) \rangle \right|^2}{\langle \Psi(z_f) , \Psi(z_f)\rangle \langle \Psi(z_i) , \Psi(z_i)\rangle} \,\, ,
\end{align}
where $\gamma$ is any path on $Z$ from $z_i$ to $z_f$ and $\mathcal{P}_\gamma$ is the parallel transport operator determined by the flat connection $\nabla$. Note that when the fundamental group $\pi_1(Z)$ of the manifold $Z$ is non-trivial, in addition to being flat, we must require that $\nabla$ has trivial holonomy, in order that probabilities depend only on initial and final states.
\smallskip

While, as discussed above, a contact manifold $(Z,\alpha)$ provides a natural Hilbert bundle~$\mathcal{H}Z$, in general one might only hope for a formal solution $\nabla^{[\hbar]}$ to the flat connection Problem~\ref{wehaveproblems}. Here, by formal, we mean an asymptotic series in $\sqrt{\hbar}$. In fact, convergence of the formal solution $\nabla^{[\hbar]}$ to a \textit{bona fide} quantum connection may necessitate a suitable truncation of the Hilbert bundle and/or specialization to distinguished values of~$\hbar$; see, for example,~\cite{chwdynamics}. The seven sphere is an ideal setting for exploring these phenomena.

The symplectic manifold analog of Problem~\ref{wehaveproblems} was originally stated and formally solved as an asymptotic expansion in $\sqrt{\hbar}$ in seminal work by Fedosov~\cite{fedosovdeform}. The generalization of that result to contact manifolds was established in~\cite{chwdynamics,Elfimov_2022}. Following Fedosov, we shall formulate an asymptotic version of Problem \ref{wehaveproblems}. Note that a connection~$\nabla$ is said to be \textit{flat to order $\ell$} when its curvature $(\nabla)^2 = \hbar^{\ell/2} F_\ell\,,$ for some smooth two-form~$F_\ell$ that is polynomial in $\sqrt{\hbar}$ and takes values in $\End(\mathcal{H}Z)$. By a \textit{formal connection} $\nabla^{[\hbar]}$ that is \textit{flat to order $\ell = \infty$}, we mean an infinite sequence of connections whose $k$-th term~$\nabla^{\hbar,k}$ is flat to order~$k$. 
\begin{problem}\label{wehaveproblems2}
	Find a formal connection $\nabla^{[\hbar]}$ that is flat up to order $\ell = \infty$ on the $U({\mathcal H})$ vector bundle~${\mathcal H}Z$ such that, for all
	$\Phi,\Psi\in \Gamma({\mathcal H}Z) $ and  $u \in \Gamma(TZ)$,
	\begin{enumerate}[(i)]
		\item\label{formalhermitean} $ \mathcal{L}_u \langle \Phi, \Psi \rangle = \langle\nabla^{[\hbar]}_u \Phi , \Psi\rangle+\langle\Phi ,\nabla^{[\hbar]}_u \Psi\rangle $, \\[-2mm]
		\item\label{formallimit} $\lim_{\hbar \to 0} i \hbar \nabla_u^{[\hbar]} \Psi = \alpha(u) \Psi$.\hfill\scalebox{1.5}{${\large\diamond}$}
	\end{enumerate}
\end{problem} 
\noindent Of course, we are requiring that items~\eqref{formalhermitean} and \eqref{formallimit} hold for every element of the sequence defined by $\nabla^{[\hbar]}$. Our first main result concerns the standard contact seven sphere discussed in Section~\ref{sec:contactsphere}.
\begin{theorem}\label{wehavesolutions2}
	Problem~\ref{wehaveproblems2} has a solution on the seven sphere $S^7$ equipped with the standard contact form obtained from its homogeneous model $\Uni(2,\HH)/\Uni(1,\HH)$.
\end{theorem} 
The proof of Theorem \ref{wehavesolutions2} depends in part on the fact that fibers of the contact distribution are 6-dimensional symplectic vector spaces that can easily be quantized in terms of $7$-dimensional Heisenberg algebra $\heis_3$ acting on the Hilbert space $L^2(\RR^3)$. We also use that $S^7$ has a (local)~$\uni(2,\HH)$ exterior differential system determined by the homogeneous model. The Fedosov-type formal deformation quantization~\cite{fedosovdeform} for contact manifolds of~\cite{Herczeg:2018hup} relies on a formal expansion of a connection form taking values in the universal enveloping algebra of the Heisenberg algebra $\heis_3$. Hence, we are interested in the embeddings of the Lie algebra~$\uni(2,\HH)$ in the Weyl algebra $\mathcal{W}(\heis_3)$. We study a formal analog of this embedding problem in Section~\ref{embeddings}, the main result of which is the embedding given in Theorem~\ref{firstembedding}.
\smallskip

We are also interested in going beyond formality. There is an extensive literature on beyond formal quantizations of observables~\cite{rieffelheisenberg,rieffeloperator,rieffelquestions,rieffelconvolution,waldmann2019,landsmancoadjoint,hawkins2008}. Here, we only focus on quantum connections on a Hilbert bundle because, over a contact manifold, they already encode Schr\"odinger equations and hence quantum dynamics. (Note that such connections can also be employed to construct quantum observables, see~\cite{fedosovdeform,fedosovbook,Herczeg:2018hup,Corradini:2020vqa}.) As discussed above, requiring that the formal connection yields a quantum connection on the Hilbert bundle~$\mathcal{H}Z$ is too strong. Instead, at least for distinguished values of $\hbar$, one can require that the formal connection yields a quantum connection on a subbundle of $\mathcal{H}Z$. For that, we say that a formal connection on $\mathcal{H}Z$ \textit{induces a quantum dynamical system} on a subbundle $\mathbb{V}Z \subset \mathcal{H}Z$, when for any $\Psi \in \Gamma(\mathbb{V}Z)$ and $u \in \Gamma(TZ)$, the limit \begin{align}\label{limitconnections}
	\nabla^{\hbar}_u \Psi := \lim_{\ell \to \infty} \nabla^{\hbar,\ell}_u \, \Psi  
\end{align} 
exists and defines a connection $\nabla^{\hbar}$ on $\mathbb{V}Z$. 

In fact, we are able to write down an explicit generating function for a formal connection $\nabla^{[\hbar]}$ solving Problem \ref{wehaveproblems2}. This involves certain square roots of operators that generically lead to a violation of the unitarity Requirement~\eqref{formalhermitean} of Problem \ref{wehaveproblems2}. These are natural generalizations of the Holstein--Primakoff mechanism~\cite{holstein}, which describes quantum spins in terms of oscillators; for similar generalizations see~\cite{Okubo_1975, Klein_Marshalek_1991, Palev_1997, Oh_Rim_1997}. However, by tuning~$\hbar$ to certain distinguished values, the fibers of the Hilbert bundle $\mathcal{H}Z$ truncate to finite-dimensional ``symmetric tensor representations'' of~$\Uni(2,\HH)$ (see Remark~\ref{re:symmetrictensors}). Moreover, the asymptotic series $\nabla^{[\hbar]}$ converges when acting on these truncated representations. Hence, on the corresponding Hilbert subbundles, we obtain \textit{bona fide} induced quantum dynamical systems.   %Our main result describing going beyond formality is the following.
\begin{theorem}\label{th:beyondformality}
	Let $\mathscr{I}:= \left\{ \frac{1}{m} \,\,\, | \,\,\, m \in \mathbb{Z}_{>0} \right\} \subset \RR \ni \hbar$. The Hilbert bundle $\mathcal{H}Z$ of Problem~\ref{wehaveproblems2} admits a filtration $$\mathcal{H}^1Z \subset \mathcal{H}^{1/2}Z \subset \mathcal{H}^{1/3}Z \subset \cdots \subset \mathcal{H}Z \, ,$$ where $(\mathcal{H}^{\hbar}Z)_{\hbar\in \mathscr{I}}$ are vector bundles associated to the homogeneous model by a family of~$\Uni(1,\HH)$-representations parametrized by $\mathscr{I}$, and obtained by restricting the symmetric tensor representations of $\uni(2,\HH)$. Moreover, when $\hbar \in \mathscr{I}$, the formal quantum connection~$\nabla^{[\hbar]}$ of Theorem~\ref{wehavesolutions2} solving Problem~\ref{wehaveproblems2} induces finite-dimensional quantum dynamical systems~$(\mathcal{H}^{\hbar}Z,\nabla^{\hbar})_{\hbar\in \mathscr{I}}$.  
\end{theorem}
\noindent 
The quantum dynamical systems $(\mathcal{H}^{\hbar}Z,\nabla^{\hbar})_{\hbar\in \mathscr{I}}$ appearing in the above theorem could have been constructed directly from the data of the homogeneous model, its Cartan connection and the symmetric tensor representations of Remark~\ref{re:symmetrictensors}. However, the fact that these quantum dynamical systems are the output of a formal quantization procedure is of primary interest. The key ingredient is Theorem~\ref{th:summarydiagram}, which shows that these representations are obtained from limits of the formal embedding. We also note that these representations are special cases of highest weight representations associated to distinguished values of $\hbar$ appearing in the coadjoint orbit quantization of~\cite{landsmancoadjoint}. In particular, the dimension of these representations becomes large in the small $\hbar$ classical limit.

\section{The contact seven sphere}\label{sec:contactsphere}
The seven sphere is the real hypersurface in $\HH^2\ni (x,y)$ satisfying 
\begin{align}\label{spherecond}
	|x|^2 + |y|^2=1 \, ,
\end{align}
where the space of quaternions ${\mathbb H}$ is a 4-dimensional associative algebra over $\RR$ with basis $\{1,\qi,\qj,\qk\}$. The multiplication for quaternions is determined by that of the quaternionic units $\vec e=(\qi,\qj,\qk)$,
$$
{\qi}^2={\qj}^2={\qk}^2=\qi \qj \qk=-1\, ,
$$  
where $1$ is the multiplicative identity.
Denoting a quaternion $q$ in $\HH$ by $q = q_0 + \vec q \cdot \vec e$, where $(q_0,\vec q)\in {\mathbb R}^4$, \textit{quaternionic conjugation} is defined by $\bar q := q_0 - \vec q\cdot \vec e$ and $|q|^2:=\bar q q$. When~$q = - \bar{q}$, we say that $q$ is \textit{purely quaternionic}. 

Calling
$$
z=x+y{\bm l}\in {\mathbb O}
$$
and $\bm m := \bm{i} \bm{l}$, $\bm n := \bm{j} \bm{l}$ and $\bm o:=\bm{k} \bm{l}$, where $\{\qi,\qj,\qk,\bm{l},\bm{m},\bm{n},\bm{o}\}$ are the standard octonionic units, the seven sphere becomes the space of unit octonions
$$
	\|z\|^2 = |x|^2 + |y|^2 = 1\, .
$$
Denoting octonionic conjugation also by a bar, the one-form
$
\frac12(\bar z dz-d\bar z z) \in \Gamma(T^*{\mathbb O})$, pulled back along the inclusion $i:S^7\hookrightarrow {\mathbb O}$,  obeys
$$
\theta:=\frac12i^*(\bar z dz-d\bar z z)=
 i^*(\bar z dz) = i^*(-d\bar z z) \, .
$$
%Hence, using the Moufang loop identity $a((bc)a)=(ab)(ca)$, it is difficult to show that\edz{CHECK.}
%Remembering that $\bar{z} z = 1 = z \bar{z}$, one has
%$$
%d \theta +  \theta \wedge \theta=0\, .
%$$
%We believe it is impossible.
Moreover, $\theta$ is purely octonionic, \textit{i.e.} $\theta = -\bar\theta$, and determines seven real global contact one-forms that parallelize $S^7$.
\smallskip

A hermitean form $\langle \pdot, \pdot \rangle$ on $\HH^2$ is defined by 
$
\langle P,Q \rangle := {P}^\dagger Q$,		
where $P$ and $Q$ denote  quaternion-valued two-component column vectors, and $\dagger$ is the quaternionic conjugate transpose operation. Also, $\|Q\|^2 := \langle Q, Q \rangle$.
The hermitean form is preserved by the quaternionic unitary group 
$$
	\Uni(2,{\mathbb H}) = \{ g \in  \GL(2,\mathbb{H}) \,\, | \,\,  g^\dagger g= \operatorname{Id}\} \, .
$$
This is a real compact Lie group of dimension $10$, see~\cite[Chapter~1]{knappbeyondliegroups}. 
Its Lie algebra is the space of quaternionic, antihermitean $2\times 2$
matrices
$$
\uni(2,\HH)=\left\{\frac12
\begin{pmatrix}
\vec J \cdot \vec e & P_0+\vec P \cdot \vec e\\
-P_0+\vec P\cdot \vec e &\vec K \cdot  \vec e
\end{pmatrix}
=:\vec J\cdot \vec \jmath + P_0  p_0+\vec P \cdot \vec p + \vec K \cdot \vec k \,\,\, \Big| \,\,\, \vec{J}, \vec{P}, \vec{K} \in \RR^3, \, P_0 \in \RR 
\right\}\, ,
$$
which is a real form of the complex symplectic Lie algebra $\sp(4,\CC)$. Indeed, $\uni(2,\HH)$ is the fixed point set of the conjugate linear involutive  automorphism $X \mapsto -X^\dagger$ of $\sp(4,\CC)$, where
\begin{align}\label{reality}
	(j_i)^\dagger = -j_i \, , \quad (p_0)^\dagger = -p_0 \, , \quad (p_i)^\dagger = -p_i \, , \quad (k_i)^\dagger = -k_i \, .
\end{align}
Here, the indices $i,j,k,\ldots$ take values $1,2,3$ and we raise and lower the indices with the Kronecker symbol $\delta_{ij}$ and its inverse.

The (non-vanishing) Lie brackets of the $\uni(2,\HH)$ generators $\{\vec \jmath,p_0,\vec p, \vec k\}$ follow from the above matrix representation and are given by
\begin{eqnarray*}
{}&	[\jmath_i,\jmath_j] =  \epsilon_{ij}{}^{k} \jmath_k \, ,\qquad \qquad  [k_i,k_j] = \epsilon_{ij}{}^{k} k_k\, ,&\\[1mm]
{}&	 [p_i,\jmath_j] =\frac12( \epsilon_{ij}{}^{k} p_k + \delta_{ij} p_0)\, ,
 \qquad 
{}	 [p_i, p_j] =  \epsilon_{ij}{}^{k}(\jmath_k +  k_k)\, ,
\qquad
{}[p_i,k_j] =\frac12( \epsilon_{ij}{}^{k} p_k - \delta_{ij} p_0)
	\,, &\\[1mm]
{}	&[p_0, \jmath_i] = -\frac12p_i \, ,\qquad \qquad [p_0,p_i] =  \jmath_i -k_i \, ,\qquad \quad \quad [p_0,k_i] =  \frac12p_i\, .&
\end{eqnarray*}
The above presentation exhibits the branching rule 
\begin{equation}\label{br343}
	\uni(2,\HH)=\uni(1,\HH)\oplus \HH \oplus \uni(1,\HH)\, 
\end{equation}
with respect to the diagonal $\Uni(1,\HH)\times \Uni(1,\HH)$ subgroup of
$\Uni(2,\HH)$, generated by $\{\vec{\jmath}, \vec{k}\}$. 

It is useful to change basis to new generators 
\begin{align}\label{spinordef}
	\scalebox{0.86}{$(J^{\alpha \beta}) := 2 \begin{pmatrix}
			-\jmath_1 + i \jmath_2 & - j_3 \\ -j_3 & \jmath_1 + i \jmath_2
		\end{pmatrix} \, , \;\; (P^\alpha_{\dot{\alpha}}) := \begin{pmatrix}
			-p_3 - i p_0 & -p_1 + i p_2  \\ p_1 + i p_2 & -p_3 + i p_0
		\end{pmatrix} \, , \;\; (K_{\dot{\alpha} \dot{\beta}}) := 2 \begin{pmatrix}
			k_1 + i k_2 & - k_3 \\ -k_3 & -k_1 + i k_2
		\end{pmatrix} \, ,$}
\end{align}
yielding
\begin{eqnarray*}
{}&	\scalebox{0.95}{[$J^{\alpha \beta},J^{\gamma \delta}] = i( \epsilon^{\alpha \gamma}J^{\beta\delta}+\epsilon^{\alpha\delta}J^{\beta\gamma}+\epsilon^{\beta \gamma}J^{\alpha\delta}+\epsilon^{\beta \delta} J^{\alpha \gamma}) $},&\\[1mm]
{}&	\hspace{-0.26 cm} \scalebox{0.908}{$[J^{\alpha \beta},P^\gamma_{\dot{\gamma}}] = i( \epsilon^{\alpha \gamma}P^\beta_{\dot{\gamma}}+\epsilon^{\beta \gamma}P^\alpha_{\dot{\gamma}}), 
	\;\; 
	{}	 [P^\alpha_{\dot{\alpha}},P^\beta_{\dot{\beta}}] = i(  \epsilon_{\dot{\alpha} \dot{\beta}}J^{\alpha \beta}+\epsilon^{\alpha \beta}K_{\dot{\alpha}\dot{\beta}})\, ,
	\;\;
	{}[K_{\dot{\alpha} \dot{\beta}},P^\gamma_{\dot{\gamma}}] = i( \epsilon_{\dot{\alpha} \dot{\gamma}}P^\gamma_{\dot{\beta}}+\epsilon_{\dot{\beta} \dot{\gamma}}P^\gamma_{\dot{\alpha}})
$}\,, &\\[1mm]
{}	&\scalebox{0.95}{$[K_{\dot{\alpha} \dot{\beta}},K_{\dot{\gamma} \dot{\delta}}] =  i(\epsilon_{\dot{\alpha} \dot{\gamma}}K_{\dot{\beta}\dot{\delta}}+\epsilon_{\dot{\alpha}\dot{\delta}}K_{\dot{\beta}\dot{\gamma}}+\epsilon_{\dot{\beta} \dot{\gamma}}K_{\dot{\alpha} \dot{\delta}}+\epsilon_{\dot{\beta} \dot{\delta}} K_{\dot{\alpha} \dot{\gamma}})$}\, ,&
\end{eqnarray*}
where
\begin{align*}
	(-\epsilon^{\alpha \beta}) = \begin{pmatrix}
		0 & -1 \\ 1 & 0
	\end{pmatrix} = (\epsilon_{\dot{\alpha} \dot{\beta}}) \, .
\end{align*}
Here and above, $\alpha\in\{+,-\}$ and $\dot{\alpha}\in\{\dot{+},\dot{-}\}$, so that $\epsilon_{\dot{+}\dot{-}}=-1$.

\smallskip

The reality conditions in Equation~\eqref{reality} now become
\begin{align}\label{spinorreality}
	\scalebox{0.87}{$(J^{++})^\dagger = J^{--}  \,, \;\;  (J^{+-})^\dagger = - J^{+-}  , \;\;  (P^+_{\dot{+}})^\dagger = - P^-_{\dot{-}} \, , \;\; (P^+_{\dot{-}})^\dagger = P^-_{\dot{+}} \,, \;\;  (K_{\dot{+}\dot{+}})^\dagger = K_{\dot{-}\dot{-}} \, , \;\; (K_{\dot{+}\dot{-}})^\dagger = -K_{\dot{+}\dot{-}} \, .$}
\end{align}
We are now ready develop the homogeneous model for the seven sphere. Useful references for homogeneous models for spheres include~\cite{montgomerysphere,borelsphere} and \cite[Chapter~7]{besseeinstein}.
\smallskip

\subsection{Homogeneous model}\label{subsec:homogeneousmodel}
The group $\Uni(2,{\mathbb H})$ acts transitively on the seven sphere $$S^7 =
\left \{ Q \in \mathbb{H}^2 \,\,\, \big| \,\,\, \|Q\|^2 = 1 \right \}\,.$$
Therefore, we view it as the homogeneous $\Uni(2,\HH)$-space 
\begin{align}
	S^7 \cong \Uni(2,\HH)/\Uni(1,\HH) \,,
\end{align}
where $\Uni(1,\HH)$ is the stabilizer of some element $o$ in $S^7$. We may choose $o = (0,1)$. Let 
\begin{align}\label{groupelement}
	g=\begin{pmatrix} w & x \\ z & y
	\end{pmatrix} \in \Uni(2,\HH) \, ,
\end{align}
so that $$ {} \left\|g \begin{pmatrix}
	0 \\ 1
\end{pmatrix} \right\|^2 = |x|^2 + |y|^2 = 1 \, ,
$$ 
which recovers the unit seven sphere condition. Moreover, $$ \begin{pmatrix}
	0 & 1
\end{pmatrix}g^{-1} dg \begin{pmatrix}
	0 \\ 1
\end{pmatrix} = 
\bar{x} \, dx + \bar{y}\, dy \, \,
$$
is a purely quaternionic one-form. This determines the first three globally defined real, contact one-forms obtained from the octonions as described above.

We would like to obtain an exterior differential system of one-forms on $S^7$ from the Maurer--Cartan one-form $\omega$ on $\Uni(2,\HH)$. The absence of global sections of the principal~$\Uni(1,\HH)$-bundle 
\begin{equation*}
	\begin{tikzcd}
		\Uni(1,\HH) \arrow[r] & \Uni(2,\HH) \arrow[d] \\ & S^7 
	\end{tikzcd}  
\end{equation*}   implies that such exists only on a local patch. One way to obtain such a patch is to assume that the quaternion $x$ appearing in $g$ is non-zero. By right multiplication with an element $h \in \Stab(o)\cong\Uni(1,\HH)$, one can bring $g$ to the form
$$
g_s=\begin{pmatrix} -\bar{x}^{-1} \bar{y} \bar{x} & x \\ \bar{x} & y
\end{pmatrix}\, .
$$
This amounts to choosing some local section $s : S^7 \setminus \{ x = 0 \} \to \Uni(2,\HH)$. Let us denote by $A_s$ the pullback $s^*\omega$ of the Maurer--Cartan form $\omega$ on $\Uni(2,\HH)$ along $s$, for which we find
\begin{align}\label{southmc}
	 A_s = \frac{1}{2}\begin{pmatrix}
		\mu & \nu  \\ - \bar{\nu} & \kappa 
	\end{pmatrix} = \vec \mu \hh \cdot \hh \vec \jmath + \nu_0  p_0+\vec \nu \hh \cdot \hh \vec p + \vec \kappa \hh \cdot \hh \vec k = \mu_{\alpha\beta}J^{\alpha\beta} + \nu^{\dot{\alpha}}_{\alpha} P^\alpha_{\dot{\alpha}} + \kappa^{\dot{\alpha}\dot{\beta}} K_{\dot{\alpha}\dot{\beta}}\, , 
\end{align}
where 
\begin{align*}
	\mu &= \frac{2}{|x|^2} \left( x d\bar{x} + x y  d\bar{y} \bar{x} \right ) + 2 \hh x y x^{-1} d (\bar{x}^{-1} ) \bar{y} \bar{x}  \, , \\
	\nu &= 2(x dy - x y x^{-1} dx) \, , \\
	\kappa &= 2(\bar{x} dx + \bar{y} dy) \,.
\end{align*}
and
\begin{eqnarray}\label{doubleindexforms}
	{}&\mu_{++} := -\frac{\mu^1 + i \mu^2}{4} \, , \qquad \mu_{+-} := -\frac{\mu^3}{4} =: \mu_{-+} \, , \qquad \mu_{--} := \frac{\mu^1 - i \mu^2}{4} &\nonumber\\[1mm]
	{}&\nu^{\dot{+}}_{+} := \frac{-\nu^3 + i \nu^0}{2} \, , \qquad \nu^{\dot{+}}_{-} := \frac{\nu^1-i\nu^2}{2} \, , \qquad \nu^{\dot{-}}_{+} := -\frac{\nu^1+i\nu^2}{2} \, , \qquad \nu^{\dot{-}}_{-} := -\frac{\nu^3 + i \nu^0}{2} \, , &\\[1mm]
	{}&\kappa^{\dot{+}\dot{+}} := \frac{\kappa^1 - i \kappa^2}{4} \, , \qquad \kappa^{\dot{+}\dot{-}} := -\frac{\kappa^3}{4} =: \kappa^{\dot{-}\dot{+}} \, , \qquad \kappa^{\dot{-}\dot{-}} := -\frac{\kappa^1 + i \kappa^2}{4} \, . &\nonumber
\end{eqnarray}
Note that the triple of one-forms defined by $\kappa$ are defined globally, and each of them is a contact form on $S^7$. Also, the sphere condition \eqref{spherecond} implies 
\begin{align*}
	\Re(\bar{x}dx + \bar{y}dy) = 0 \, .
\end{align*} 

The Maurer--Cartan equation
$ d\omega + \omega \wedge \omega = 0  $ leads to the exterior differential system
\begin{align}
	&d\mu^i + \frac{1}{2}\epsilon^i{}_{jk} \left(\mu^j\wedge\mu^k + \nu^j \wedge \nu^k\right) +  \nu^0 \wedge \nu^i = 0 \, , \nonumber\\
	&d\nu^i + \frac{1}{2} \epsilon^i{}_{jk} \left(\mu^j\wedge\nu^k + \kappa^j \wedge \nu^k\right) - \frac{1}{2}\nu^0 \wedge (\mu^i - \kappa^i) = 0\, , \nonumber\\[-2.6 mm] \label{exteriorsystem} \\[-2.6 mm]
	&d\nu^0 + \frac{1}{2}\delta_{ij} \, \nu^i \wedge (\mu^j-\kappa^j) = 0\, , \nonumber\\
	&d\kappa^i + \frac{1}{2}\epsilon^i{}_{jk} \left(\kappa^j\wedge\kappa^k + \nu^j \wedge \nu^k\right) - \nu^0 \wedge \nu^i = 0 , \nonumber
\end{align}
for the 10 one-forms $(\vec{\mu},\nu^0,\vec{\nu},\vec{\kappa})$ on $S^7 \setminus \{ x = 0\}$.
 
To globally describe the Maurer--Cartan form pulled-back to $S^7$, we show how to pass from a patch where $x$ is non-vanishing to one where $y$ is non-vanishing. For the latter, we have $$g_n = \begin{pmatrix}
	\bar{y} & x \\ - \bar{y}^{-1} \bar{x} \bar{y} & y 
\end{pmatrix} \,,$$
where the corresponding choice of section $n : S^7 \setminus \{y = 0\} \to \Uni(2,\HH)$ was obtained by right multiplying the group element $g$ of Equation~\eqref{groupelement} with
$$ \begin{pmatrix}
	w^{-1}\bar{y} & 0 \\ 0 & 1
\end{pmatrix} \,.$$
(Because $y$ is invertible so too is $w$.) 

On the overlap of the two patches, the group elements $g_s$ and~$g_n$ are related by right multiplication by an element $h$ of the structure group $\Uni(1,\HH)$: $$g_n = g_s \hh h \, , \quad \mbox{ where }\,  h = \begin{pmatrix}
	\tau & 0 \\ 0 & 1
\end{pmatrix} \, \mbox{ with } \, \tau := - \bar{x}^{-1} \bar{y}^{-1} \bar{x} \bar{y} \, \, .$$ The pullback of the Maurer-Cartan form along the section $n$ is then given by
$$ A_n =  g_n^{-1} dg_n = \frac{1}{2}\begin{pmatrix}
	\mu_n & \nu_n  \\ - \bar{\nu}_n & \kappa_n
\end{pmatrix} = h^{-1} \hh A_s \hh h + h^{-1} dh \, ,$$
where \begin{align*}
	\frac{\mu_n}{2} = \bar{\tau}\hh \frac{\mu}{2} \hh \tau + \bar{\tau} d\tau = \frac{1}{|y|^2} \left( y d\bar{y} + y x  d\bar{x} \bar{y} \right ) + y x y^{-1} d (\bar{y}^{-1} ) \bar{x} \bar{y}   \, , \, \quad
	\nu_n = \bar{\tau} \nu = 2(y dx - y x y^{-1} dy) \, .
\end{align*}
Note that as $\kappa$ is global, we have $\kappa_n=\kappa$. Also, since the homogeneous model is reductive (see below), $\frac{\mu}{2}$ is the pullback of the principal~$\Uni(1,\HH)$-connection obtained by projecting the Maurer--Cartan form~$\omega$ onto the $\uni(1,\HH)$-component generated by $\vec{\jmath}$. Therefore (as seen above) it transforms as a connection with respect to the $\Uni(1,\HH)$ element~$\tau$.
\smallskip

We note that the above treatment of $S^7$ as a homogeneous space is a special case of a well-established more general theory (see, for example, \cite{sharpe,kobayashinomizu2,cap2009parabolic}): Recall that a homogeneous space $G/H$ is said to be reductive~\cite[Chapter~4]{sharpe} when the Lie algebra $\frg$ of $G$ has a direct sum decomposition $\frg = \frh \oplus \frm$, where $\frh$ is the Lie algebra of $H$ and $$\Ad_H \frm \subseteq \frm\, .$$ In that case, the projection $\omega_{\frh}$ of the Maurer--Cartan form $\omega$ on $G$ onto $\frh$ is a principal connection because it gives a equivariant direct sum decomposition of $TG$ into vertical and horizontal parts. Furthermore, there is a one-to-one correspondence between horizontal $H$-equivariant differential $k$-forms on $G$ with values in a $G$-representation $\mathbb{V}$ and differential $k$-forms on $G/H$ with values in the associated vector bundle $G \times_H \mathbb{V}$~\cite{kobayashinomizu1,cap2009parabolic}.  In other words, \begin{align}\label{equivalencedifforms}
	\Omega^k_{\omega_{\frh}}(G,\mathbb{V})^H \cong \Omega^k(G/H, G \times_H \mathbb{V})\, .
\end{align} Also, $\omega$ defines a principal connection one-form $\hat{\omega}$ on the extended principal $G$-bundle $\hat{G} := G \times_H G$ over $G/H$ (see for instance \cite[Chapter~1]{cap2009parabolic}) and $\hat{G} \times_G \mathbb{V} \cong G \times_H \mathbb{V}$, so $$\Omega^k_{\hat{\omega}}(\hat{G},\mathbb{V})^G \cong \Omega^k(G/H, G \times_H \mathbb{V})\,.$$ Returning to $S^7 \cong \Uni(2,\HH) / \Uni(1,\HH)$, the pullback of $\omega_\frh$ along $s$ gives $\mu$. Also, when $k=1$ and $\mathbb{V} = \RR$ the trivial representation of $H = \Uni(1,\HH)$,  Equation~\eqref{equivalencedifforms} shows that the $\uni(1,\HH)$-component (the one generated by $K_{\dot{\alpha}\dot{\beta}}$) of $\omega$ on $G = \Uni(2,\HH)$ gives globally defined three one-forms $\kappa^1$, $\kappa^2$, $\kappa^3$ on $S^7$, which are all contact.

\subsection{Reeb dynamics}
The integral curves of the Reeb vector fields of the three contact forms $\kappa^i$ canonically determine dynamical trajectories. These vector fields are particularly simple when the seven sphere is decomposed into tori. For that, let $$ x = r_1 e^{-\qk \theta_1} + \qj r_2 e^{-\qk \theta_2}\,, \qquad y = r_3 e^{-\qk \theta_3} + \qj r_4 e^{-\qk \theta_4}\, ,$$ 
where $$r_1^2 + r_2^2 + r_3^2 + r_4^2 = 1\,, $$
and $\theta_1,\theta_2,\theta_3,\theta_4 \in [0,2\pi) $, $r_1,r_2,r_3,r_4 \in [0,1]$. The contact form~$ \alpha := - \frac{\kappa^3}{2}$ then becomes 
\begin{align*}
%	\kappa^1 &= \frac{1}{2} \left( \cos(\theta_1+\theta_2)\cos^2(\varphi)(d\theta_1-d\theta_2) + \cos(\theta_3+\theta_4)\sin^2(\varphi)(d\theta_3-d\theta_4 \right) \, ,\\
%	\kappa^2 &= \frac{1}{2} \left( \sin(\theta_1+\theta_2)\cos^2(\varphi)(d\theta_1-d\theta_2) + \sin(\theta_3+\theta_4)\sin^2(\varphi)(d\theta_3-d\theta_4 \right) \, ,\\
	\alpha &= r_1^2 d\theta_1 + r_2^2 d\theta_2 + r_3^2 d\theta_3 + r_4^2 d\theta_4 \, .
\end{align*} 
The corresponding Reeb vector field is
\begin{align*}
	%R_1 &= \frac{1}{\cos(\theta_1+\theta_2)}\left(-\frac{\partial}{\partial \theta_1} + \frac{\partial}{\partial \theta_2}\right) + \frac{1}{\cos(\theta_3+\theta_4)}\left(-\frac{\partial}{\partial \theta_3} + \frac{\partial}{\partial \theta_4}\right) \, , \\
	%R_2 &= \frac{1}{\sin(\theta_1+\theta_2)}\left(\frac{\partial}{\partial \theta_1} - \frac{\partial}{\partial \theta_2}\right) + \frac{1}{\sin(\theta_3+\theta_4)}\left(\frac{\partial}{\partial \theta_3} - \frac{\partial}{\partial \theta_4}\right) \, , \\
	R &=  \frac{\partial}{\partial \theta_1} + \frac{\partial}{\partial \theta_2} + \frac{\partial}{\partial \theta_3} +\frac{\partial}{\partial \theta_4} \, \, 
\end{align*} 
and Reeb orbits are periodic.

\section{Formal embedding}\label{embeddings}

The contact distribution $\xi$ is a symplectic vector bundle over the contact manifold~$S^7$. In particular, its fibers are~6-dimensional symplectic vector spaces with respect to the Levi two-form $d\alpha$ restricted to $\xi$. The Heisenberg algebra $\heis_3$ is a central extension of these fiberwise symplectic vector spaces viewed as trivial Lie algebras. The Lie algebra~$\heis_3$ will control solutions to the formal deformation quantization Problem~\ref{wehaveproblems2} for~$7$-dimensional contact manifolds. For that, we need to consider its universal enveloping algebra~$\mathcal{U}(\heis_3)$. As an algebra,~$\mathcal{U}(\heis_3)$ is the free associative algebra over $\CC$, with generators~$I, a, a^\dagger, a^\alpha_{ \dot{\alpha}}$, quotiented by the ideal generated by $$ a \hh I - I \hh a \, , \quad a^\dagger \hh I - I \hh a^\dagger \, , \quad a a^\dagger - a^\dagger a - I \, , \quad a^\alpha_{ \dot{\alpha}} a^\beta_{ \dot{\beta}} - a^\beta_{ \dot{\beta}} a^\alpha_{ \dot{\alpha}} +  \epsilon^{\alpha\beta}\eta_{\dot{\alpha}\dot{\beta}} \hh I \, , \quad a^\alpha_{ \dot{\alpha}} \hh I - I \hh a^\alpha_{ \dot{\alpha}} \, \, .$$ 
It follows that $\mathcal{U}(\heis_3)$ is a Lie algebra with non-vanishing brackets given by commutators
\begin{align*}
	[a, a^\dagger] = I \, , \qquad [a^\alpha_{ \dot{\alpha}}, a^\beta_{ \dot{\beta}}] = - \epsilon^{\alpha\beta}\eta_{\dot{\alpha}\dot{\beta}} \, .
\end{align*}
The Heisenberg algebra $\heis_3$ is the Lie subalgebra of $\mathcal{U}(\heis_3)$ spanned by $\langle I, a, a^\dagger, a^\alpha_{\dot{\alpha}} \rangle$. In what follows, we will be primarily concerned with the unital Weyl algebra $\mathcal{W}(\heis_3)$ over~$\CC$, where the generator $I$ is realized by the identity element $1$.

We may equip $\mathcal{W}(\heis_3)$ with a conjugate linear involutive automorphism $\dagger$, where  
$$ (a)^\dagger = a^\dagger \, , \quad (a^+_{\dot{+}})^\dagger = -a^-_{\dot{-}} \, , \quad (a^+_{\dot{-}})^\dagger = a^-_{\dot{+}} \, ,$$
and for all $x,y \in \mathcal{W}(\heis_3)$,  $(xy)^\dagger := y^\dagger x^\dagger$. Eventually,
the above rules will follow from the reality conditions of Equation~\eqref{reality}.

Let us sketch our formal quantum connection construction (see Subsection~\ref{sec:formalquantization} for details) in order to better motivate the formal embedding problem. The globally defined contact one-form $\alpha = 2 \hh \kappa^{\dot{+}\dot{-}}$ and the central element~$1 \in \mathcal{W}(\heis_3)$ will determine a $\Uni(1)$-connection
\begin{align}\label{minustwonabla}
	\nabla^{\hbar,-2} = \frac{i}{\hbar} \alpha \, 1 + d \, ,
\end{align}
acting on the sections of the Hilbert bundle $\mathcal{H}Z$. We consider a grading with respect to $\sqrt{\hbar}$ so that~$\frac{1}{\hbar}$ has grade $-2$. The connection $\nabla^{\hbar,-2}$ is flat to order $\ell = -2$, meaning its curvature $$(\nabla^{\hbar,-2})^2 = \frac{i}{\hbar} \, d\alpha \, 1\, .$$
We next seek a connection satisfying Items~\eqref{formalhermitean} and~\eqref{formallimit} of Problem~\ref{wehaveproblems2} that is flat to order~$\ell = -1$. This is achieved by adding one-forms valued in $\End(\mathcal{H}Z)$ to $\nabla^{\hbar,-2}$ in order to cancel the curvature $(\nabla^{\hbar,-2})^2$. Employing the exterior system of Equation~\eqref{exteriorsystem} leads to a new connection
\begin{align}\label{minusonenabla}
	\nabla^{\hbar,-1} = \frac{i}{\hbar} \alpha \, 1 +  \frac{1}{\sqrt{\hbar}} \left( - 2 i \, \kappa^{\dot{+} \dot{+}} \, a + 2 i \, \kappa^{\dot{-} \dot{-}} \, a^\dagger  + \, \nu^{\dot{+}}_{+}\, a^{+}_{\dot{+}} + \, \nu^{\dot{-}}_{-} \, a^{-}_{\dot{-}} + \, \nu^{\dot{+}}_{-}\, a^{-}_{\dot{+}} + \, \nu^{\dot{-}}_{+}\, a^{+}_{\dot{-}}\right) + d \,
\end{align}
that is flat to order $\ell=-1$. To write the above formula, we used the one-forms defined by Equation~\eqref{southmc} in the local patch on $S^7$ corresponding to the section $s$, but could have equally well used the section $n$.

One can now try to continue along these lines to achieve higher order asymptotic flatness. For that, consider the ansatz
\begin{align}\label{ansatzconnection}
	\nabla^{[\hbar]} = d + \kappa^{\dot{\alpha}\dot{\beta}}\, \mathcal{K}_{\dot{\alpha} \dot{\beta}} + \nu^{\dot{\alpha}}_\alpha \, \mathcal{P}^\alpha_{ \dot{\alpha}} + \mu_{\alpha\beta} \, \mathcal{J}^{\alpha\beta} = : d + \mathcal{A} \, ,
\end{align}
where $$\mathcal{K}_{\dot{+}\dot{-}} = \frac{i}{\hbar} \, 1 + \cdots \, \, \, .$$ 
Here, the $\cdots$'s denote terms of higher grades in $\sqrt{\hbar}$ taking values in $\mathcal{W}(\heis_3)$. The triple $\mathcal{K}_{\dot{\alpha}\dot{\beta}}$,  $\mathcal{P}^{\alpha}_{\dot\alpha}$, $\mathcal{J}^{\alpha\beta}$ are also elements of the ring of formal Laurent series in $\sqrt{\hbar}$ with coefficients in $\mathcal{W}(\heis_3)$. Since the curvature of $\nabla^{[\hbar]}$ is now governed by the $\uni(2,\HH)$ exterior differential system of Equation~\eqref{exteriorsystem}, flatness to order~$\ell = \infty$ can be achieved if the $\mathcal{W}(\heis_3)$-valued differential form coefficients (formally) generate~$\uni(2,\HH)$. Hence, we seek formal embeddings of $\uni(2,\HH)$ into this ring.

To be more precise, let us denote the space of formal Laurent series in the indeterminate~$\sqrt{\hbar}$ with coefficients in a ring $R$ by $$R^{\hbar} := R\big[\sqrt{\hbar}^{-1},\sqrt{\hbar}\hh\big]\big] \ni \sum_{k=M}^{\infty} r_k \sqrt{\hbar}^k,$$ where $r_k \in R$ and $M$ is some integer. Two such series equal one another if and only if the sequences formed by their coefficients are the same. Note that $R^{\hbar}$ is a ring with addition and multiplication defined in the standard way. We are particularly interested in the case $R = \mathcal{W}(\heis_3)$, which is both a ring and a $\CC$-module. In turn, $$\mathcal{W}^{\hbar} = \mathcal{W}(\heis_3)[\sqrt{\hbar}^{-1},\sqrt{\hbar}]]\,$$
is a $\CC^{\hbar}$-module. Moreover, because $\mathcal{W}(\heis_3)$ is an associative $\CC$-algebra, $\mathcal{W}^{\hbar}$ is an associative $\CC^{\hbar}$-algebra. Thus, it is a Lie algebra with brackets given by commutators. The dagger operation is extended to $\mathcal{W}^{\hbar}$ in the obvious way, in particular, $\sqrt{\hbar}^\dagger := \sqrt{\hbar}$. An algebra \textit{embedding} is an injective algebra homomorphism between algebras. When the codomain is an algebra of formal Laurent series, we will call such a map a \textit{formal embedding}. We also use this terminology for Lie algebras. We begin by examining formal embeddings in the simpler case of $\uni(1,\HH)$, and then generalize to the case of interest. 

\subsection{Holstein--Primakoff mechanisms}\label{sec:holsteinprimakoff}
In~\cite{chwdynamics}, the three-sphere $S^3 \cong \Uni(1,\HH)$ was quantized by formally embedding $\uni(1,\HH)$ into $\mathcal{W}(\heis_1)[\sqrt{\hbar}^{-1},\sqrt{\hbar}]]$. This embedding in fact dates back to work on ferromagnetism by Holstein--Primakoff~\cite{holstein}. 
\medskip

The Lie algebra $\uni(1,\HH)$
$$
[k_i,k_j]=\epsilon_{ijk}k_k
$$
admits an $\dot{\rm I}$n\"on\"u--Wigner contraction~\cite{inonuwigner} to the Heisenberg algebra. For that, we replace~$k_3$ by $-\lambda^2 \mathsf{I}$, $k_{1}$ by $\frac{\lambda}{\sqrt{2}}(\zeta_1-\zeta_2)$, and $k_{2}$ by $\frac{\lambda}{\sqrt{2}i} (\zeta_1+\zeta_2)$ and then sending the real parameter~$\lambda \to \infty$, whence
$$
[\mathsf{I} , \zeta_1] = 0 = [\zeta_{2}, \mathsf{I}] \, ,\quad
[\zeta_1 , \zeta_2] = -i \hh \mathsf{I}\, .
$$
The resulting Lie algebra is the 3-dimensional Heisenberg algebra $\heis_1$. In particular, note that the generator $k_3$ of $\uni(1,\HH)$ contracts to the central element $\mathsf{I}$. This Heisenberg algebra can be embedded into $\mathcal{W}(\heis_1)[\sqrt{\hbar}^{-1},\sqrt{\hbar}]]$ by the Lie algebra homomorphism \begin{align*}
	\mathsf{I} \mapsto i \frac{1}{\hbar} \, , \qquad \zeta_1 \mapsto -i \frac{a}{\sqrt{\hbar}} \, , \qquad  \zeta_2 \mapsto i \frac{a^\dagger}{\sqrt{\hbar}} \, \, .
\end{align*} 
Holstein and Primakoff answered the question of whether there exist expansions
$$
K_+= -i \frac{a}{\sqrt{\hbar}} + \cdots\, ,\qquad
K_0 = i \left(\frac{1}{\hbar} - 2 a^\dagger a - \frac12 \right)\, ,\qquad
K_-= i \frac{a^\dagger}{\sqrt{\hbar}} +\cdots \, ,
$$
where the $\cdots$'s denote formal power series in $\sqrt{\hbar}$ taking values in the Weyl algebra, such that
$$
[K_0,K_{\pm}]= \pm 2i K_{\pm}\, ,\qquad
 [K_+,K_-]= - i K_0\, . 
$$
Remarkably, they found the following solution~\cite{holstein}
$$
K_+=  -i \sqrt{1 - \hbar \left( a^\dagger a + \frac{1}{2}\right)}\frac{a}{\sqrt{\hbar}} \, , \qquad K_- = i \frac{a^\dagger}{\sqrt{\hbar}}\sqrt{1 - \hbar \left( a^\dagger a + \frac{1}{2}\right)}
\, .
$$
Formally expanding the square roots in powers of $\hbar$ yields an embedding of $\uni(1,\HH)$ into the ring of formal Laurent series in $\sqrt{\hbar}$ with coefficients in the Weyl algebra~$\mathcal{W}(\heis_1)$. Moreover, Holstein and Primakoff showed that, for distinguished values of $\hbar$, these formal power series lead to operators giving unitary irreducible spin representations of $\uni(1,\HH)$ on finite-dimensional subspaces of a Fock space. 
\smallskip

We want to generalize the above ``Holstein--Primakoff mechanism'' to $\uni(2,\HH)$. In that case, there is a contraction to a Heisenberg algebra corresponding to the branching rule in Equation~\eqref{br343}. For that, we introduce the (invertible) invariant tensor 
$$(\eta_{\dot{\alpha}\dot{\beta}})=\begin{pmatrix}
	0&1\\1&0
\end{pmatrix} = (\eta^{\dot{\alpha}\dot{\beta}})$$
defining the $\uni(1,\CC)$ subalgebra of $\uni(1,\HH)$ generated by $K_{\dot{+}\dot{-}} := \frac{1}{2} \eta^{\dot{\alpha}\dot{\beta}} K_{\dot{\alpha}\dot{\beta}}$. Let us denote the symmetric trace-free part of a tensor $X_{\dot{\alpha}\dot{\beta}}$ by $$X_{( \dot{\alpha} \dot{\beta})_\circ} := X_{(\dot{\alpha} \dot{\beta})} - \frac{1}{2} \eta^{\dot{\gamma} \dot{\delta}}X_{\dot{\gamma} \dot{\delta}}\, \eta_{\dot{\alpha}\dot{\beta}} \, .$$ 
A $\dot{\rm I}$n\"on\"u--Wigner contraction is performed by making the replacements
\begin{align}\label{u2contraction}
	\jmath^{\alpha \beta}:= J^{\alpha \beta},\quad \pi^\alpha_{ \dot{\alpha}}:= \lambda \hh P^\alpha_{ \dot{\alpha}},\quad \zeta_{(\dot{\alpha} \dot{\beta})_\circ} := \lambda \hh K_{( \dot{\alpha} \dot{\beta})_\circ}  \, , \quad  \zeta_{\dot{+} \dot{-}} := \mathsf{I} = \lambda^2 K_{\dot{+} \dot{-}}  \,\, .
\end{align}
Sending the real parameter $\lambda\to \infty $ yields the following non-vanishing brackets
\begin{eqnarray*}
{}&	[\jmath^{\alpha \beta},\jmath^{\gamma \delta}] =  i (\epsilon^{\alpha \gamma}\jmath^{\beta\delta}+\epsilon^{\alpha\delta}\jmath^{\beta\gamma}+\epsilon^{\beta \gamma}\jmath^{\alpha\delta}+\epsilon^{\beta \delta} \jmath^{\alpha \gamma}) \, ,&\\[1mm]
{}&	 [\jmath^{\alpha \beta},\pi^{\gamma}_{\dot{\gamma}}] = i(\epsilon^{\alpha \gamma}\pi^{\beta}_{\dot{\gamma}}+\epsilon^{\beta \gamma}\pi^{\alpha}_{\dot{\gamma}}) \, ,
 \qquad 
{}	 [\pi^\alpha_{\dot{\alpha}},\pi^\beta_{ \dot{\beta}}] =  i \epsilon^{\alpha \beta}\eta_{\dot{\alpha}\dot{\beta}}\;\mathsf{I}\, 
, &\\[1mm]
{}	&[\zeta_{(\dot{\alpha} \dot{\beta})_\circ},\zeta_{(\dot{\gamma} \dot{\delta})_\circ}] = i (\epsilon_{\dot{\alpha} \dot{\gamma}}\eta_{\dot{\beta}\dot{\delta}}+\epsilon_{\dot{\alpha}\dot{\delta}}\eta_{\dot{\beta}\dot{\gamma}}+\epsilon_{\dot{\beta} \dot{\gamma}}\eta_{\dot{\alpha} \dot{\delta}}+\epsilon_{\dot{\beta} \dot{\delta}} \eta_{\dot{\alpha} \dot{\gamma}})\;\mathsf{I} \, . &
\end{eqnarray*}
In particular, the generator $\mathsf{I}$ is central in the large $\lambda$ limit. The above algebra is the Lie algebra $\heis_3 \oplus \uni(1,\HH)$. Our goal is to find a realization of $\uni(2,\HH)$ in $\mathcal{W}^{\hbar}$ starting from a realization of the above Lie algebra $\heis_3 \oplus \uni(1,\HH)$. 
\subsection{Poisson algebra embedding}\label{sec:poissonalgebraembedding}
Before presenting our main formal embedding result in Theorem~\ref{firstembedding}, we consider a classical version of the problem. For that, we ask whether we can find functions $\mathcal{J}^{\alpha\beta}$, $ \mathcal{P}^\alpha_{ \dot{\alpha}}$, $ \mathcal{K}_{\dot{\alpha}\dot{\beta}}$ of $z$ and $z^\alpha_{\dot{\alpha}}$ taking values in $\CC^{\hbar}$ that realize the Lie algebra $\uni(2,\HH)$ as a Lie subalgebra of a Poisson algebra with brackets $\{\pdot, \pdot\}$ such that
\begin{align*}
	\{z, \bar{z}\} = -i \, , \qquad \{z^\alpha_{ \dot{\alpha}}, z^\beta_{ \dot{\beta}}\} = i \epsilon^{\alpha\beta}\eta_{\dot{\alpha}\dot{\beta}} \, \, .
\end{align*}
It is easy to realize the $\uni(1,\HH)$ subalgebra
$$	\scalebox{0.902}{$\{\mathcal{J}^{\alpha \beta},\mathcal{J}^{\gamma \delta}\} = \epsilon^{\alpha \gamma}\mathcal{J}^{\beta\delta}+\epsilon^{\alpha\delta}\mathcal{J}^{\beta\gamma}+\epsilon^{\beta \gamma}\mathcal{J}^{\alpha\delta}+\epsilon^{\beta \delta} \mathcal{J}^{\alpha \gamma}$} \, ,$$
by setting
$$ 
\mathcal{J}^{++} = -2i z^+_{\dot{-}} z^+_{\dot{+}}  \, \, , \qquad \mathcal{J}^{+-} = -i ( z^+_{\dot{+}} z^-_{\dot{-}} + z^-_{\dot{+}} z^+_{\dot{-}} ) = \mathcal{J}^{-+}  \,\, , \qquad  \mathcal{J}^{--} = -2i z^-_{\dot{+}} z^-_{\dot{-}}  \,\, .
$$
Now, we make an \textit{ansatz} for the remaining Lie algebra elements:
\begin{align*}
	\mathcal{K}_{\dot{+}\dot{+}} = -2i \frac{z}{\sqrt{\hbar}} + \cdots \, , \qquad \mathcal{K}_{\dot{+}\dot{-}} = \mathcal{K}_{\dot{-}\dot{+}} = i \frac{1}{\hbar} + \cdots\, , \qquad  \mathcal{K}_{\dot{-}\dot{-}} = 2i \frac{\bar{z}}{\sqrt{\hbar}} + \cdots \, ,
\end{align*}
and
\begin{align*}
	\mathcal{P}^\alpha_{\dot{\alpha}} = \frac{z^\alpha_{\dot{\alpha}}}{\sqrt{\hbar}} + \cdots \, ,
\end{align*}
where the $\cdots$'s denote higher order terms in $\sqrt{\hbar}$. We must now require that the above functions obey the following Poisson brackets,
\begin{eqnarray*}
	{}&	 \hspace{-0.19 cm}\scalebox{0.902}{$\{\mathcal{J}^{\alpha \beta},\mathcal{P}^\gamma_{\dot{\gamma}}\} = \epsilon^{\alpha \gamma}\mathcal{P}^\beta_{\dot{\gamma}}+\epsilon^{\beta \gamma}\mathcal{P}^\alpha_{\dot{\gamma}}, 
		\;\; 
		{}	 \{\mathcal{P}^\alpha_{\dot{\alpha}},\mathcal{P}^\beta_{\dot{\beta}}\} =  \epsilon_{\dot{\alpha} \dot{\beta}}\mathcal{J}^{\alpha \beta}+\epsilon^{\alpha \beta}\mathcal{K}_{\dot{\alpha}\dot{\beta}}\, ,
		\;\;
		{}\{\mathcal{K}_{\dot{\alpha} \dot{\beta}},\mathcal{P}^\gamma_{\dot{\gamma}}\} =  \epsilon_{\dot{\alpha} \dot{\gamma}}\mathcal{P}^\gamma_{\dot{\beta}}+\epsilon_{\dot{\beta} \dot{\gamma}}\mathcal{P}^\gamma_{\dot{\alpha}}$} \,,
	&\\[1mm]
	{}	&\scalebox{0.902}{$\{\mathcal{K}_{\dot{\alpha} \dot{\beta}},\mathcal{K}_{\dot{\gamma} \dot{\delta}}\} =  \epsilon_{\dot{\alpha} \dot{\gamma}}\mathcal{K}_{\dot{\beta}\dot{\delta}}+\epsilon_{\dot{\alpha}\dot{\delta}}\mathcal{K}_{\dot{\beta}\dot{\gamma}}+\epsilon_{\dot{\beta} \dot{\gamma}}\mathcal{K}_{\dot{\alpha} \dot{\delta}}+\epsilon_{\dot{\beta} \dot{\delta}} \mathcal{K}_{\dot{\alpha} \dot{\gamma}}$}\, .&
\end{eqnarray*}
Note that the functions $n := 2\bar{z} z$ and $N :=  z^+_{\dot{-}}z^-_{\dot{+}} - z^+_{\dot{+}}z^-_{\dot{-}} $ have vanishing Poisson brackets with the $\uni(1,\HH)$ generators $\mathcal{J}^{\alpha\beta}$. Therefore, we refine our \textit{ansatz} to read
\begin{eqnarray*}
	&\mathcal{K}_{\dot{+}\dot{+}} = - 2i e^{ i \psi} F^{\hbar}(n,N) \, z \, , \quad \mathcal{K}_{\dot{+}\dot{-}} = i \left( \frac{1}{\hbar} - n - N \right) = \mathcal{K}_{\dot{-}\dot{+}} \, , \quad\mathcal{K}_{\dot{-}\dot{-}} = 2 i \bar{z} \, e^{ -i \psi} F^{\hbar}(n,N)  \,\, ,&\\[2 mm]
	&\scalebox{0.93}{$ \hspace{-0.3 cm}\mathcal{P}^+_{\dot{+}} = - e^{ -i \phi_2} H^{\hbar}_2(n,N) z^+_{\dot{-}} \hh z \, + \, e^{ i \phi_1}H^{\hbar}_1(n,N) z^+_{\dot{+}}\, , \quad \mathcal{P}^-_{\dot{-}} =  \bar{z} \, z^-_{\dot{+}} \, e^{i \phi_2} H^{\hbar}_2(n,N)  +  z^-_{\dot{-}} e^{ - i\phi_1}\, H^{\hbar}_1(n,N) \,\, ,$}& \\[1.5 mm]
	&\scalebox{0.93}{$ \hspace{-0.4 cm} \mathcal{P}^+_{\dot{-}} =  \bar{z} \, z^+_{\dot{+}} \, e^{ i \varphi_2} G^{\hbar}_2(n,N)  + z^+_{\dot{-}} \, e^{ i \varphi_1}G^{\hbar}_1(n,N)\, , \quad \,\,\,\, \,\,\,\,\,\, \mathcal{P}^-_{\dot{+}} =  - e^{ -i \varphi_2}G^{\hbar}_2(n,N) z^-_{\dot{-}} \hh z +  e^{ -i \varphi_1}G^{\hbar}_1(n,N) \; z^-_{\dot{+}}\,\, .$}&
\end{eqnarray*} 
Here, we have made a polar decomposition such that the functions $F^{\hbar}$, $H_1^{\hbar}$, $H_2^{\hbar}$, $G_1^{\hbar}$, $G_2^{\hbar}$, $\psi$, $\phi_1$, $\phi_2$, $\varphi_1$, $\varphi_2$ are real-valued functions of $n$ and $N$.  
Computing the Poisson brackets of these remaining generators and requiring that these obey the above realization of~$\uni(2,\HH)$ yields a system of differential equations. A family of solutions thereof is given below:
\begin{eqnarray}
	&\mathcal{K}_{\dot{+}\dot{+}} = - 2i e^{ i \psi} \sqrt{\frac{1}{\hbar} - \frac{n}{2} - N} \hh z \, , \, \, \, \, \, \hh \mathcal{K}_{\dot{+}\dot{-}} = i \left( \frac{1}{\hbar} - n - N \right) = \mathcal{K}_{\dot{-}\dot{+}}\, , \, \, \, \, \, \hh \mathcal{K}_{\dot{-}\dot{-}} =  2i \bar{z} \hh e^{ -i \psi} \sqrt{\frac{1}{\hbar} - \frac{n}{2} - N}  \, ,& \nonumber \\[1 mm]
	&\mathcal{P}^+_{\dot{+}} = - e^{- i \phi_2}  z^+_{\dot{-}} \, z +   e^{ i \phi_1}\sqrt{\frac{1}{\hbar} - \frac{n}{2} - N} \, z^+_{\dot{+}} , \quad \mathcal{P}^-_{\dot{-}} =    \bar{z} \, z^-_{\dot{+}} \, e^{ i \phi_2} +  z^-_{\dot{-}} \hh e^{ - i\phi_1} \sqrt{\frac{1}{\hbar} - \frac{n}{2} - N} \,\, ,& \nonumber \\[-2 mm] \label{classicalbracketsolution} \\[-2 mm] 
	& \, \, \mathcal{P}^+_{\dot{-}} =  \bar{z} \, z^+_{\dot{+}} \, e^{ i \phi_2}     + z^+_{\dot{-}} \, e^{ -i \phi_1}\sqrt{\frac{1}{\hbar} - \frac{n}{2} - N}\, , \quad \mathcal{P}^-_{\dot{+}} =  - e^{ -i \phi_2} z^-_{\dot{-}} \, z   +  e^{ i \phi_1}\sqrt{\frac{1}{\hbar} - \frac{n}{2} - N} \; z^-_{\dot{+}}\,\, ,& \nonumber \\[1.5 mm]
	& \mathcal{J}^{++} = -2i z^+_{\dot{-}} z^+_{\dot{+}}  \,\, , \qquad
	\mathcal{J}^{+-} = -i ( z^+_{\dot{+}} z^-_{\dot{-}} + z^-_{\dot{+}} z^+_{\dot{-}} ) = \mathcal{J}^{-+}  \,\, , \qquad  \mathcal{J}^{--} = -2i z^-_{\dot{+}} z^-_{\dot{-}}  \,\, , \nonumber&
\end{eqnarray}  
where the phases obey $\psi = \phi_1 - \phi_2$ and $\partial_N \psi = 2 \hh \partial_n \phi_1$. The map 
$$ J^{\alpha\beta} \mapsto \mathcal{J}^{\alpha\beta}\, , \qquad P^\alpha_{ \dot{\alpha}} \mapsto \mathcal{P}^\alpha_{ \dot{\alpha}} \, , \qquad K_{\dot{\alpha}\dot{\beta}} \mapsto \mathcal{K}_{\dot{\alpha}\dot{\beta}}\, , $$
gives an embedding of $\uni(2,\HH)$ into a space of functions of $z$ and $z^{\alpha}_{\dot{\alpha}}$ taking values in $\CC^{\hbar}$.

We can already perform a \textit{na\"ive} quantization by first setting the phases to unity, \textit{i.e.} $\psi = \phi_1 = \phi_2 = 0$, and then defining a map
$$ z \mapsto a \, , \qquad \bar{z} \mapsto a^\dagger \, , \qquad z^\alpha_{\dot{\alpha}} \mapsto a^\alpha_{\dot{\alpha}} \, , $$
that directly replaces the classical variables, in the orders written above, with the corresponding elements of the Weyl algebra. In particular, we directly replace functions of~$n,N$ by corresponding elements in $\mathcal{W}^{\hbar}$, where $n$ and~$N$ are now the $\mathcal{W}(\heis_3)$ elements defined in Equations~\eqref{smalln} and \eqref{bign} below. In general, such a procedure is, of course, fraught~\cite{Groenewold,gotays2}. 
Remarkably it here reproduces the formal quantum solution given in Theorem \ref{firstembedding} below.

\subsection{Formal quantum embedding}\label{sec:firstembedding} We would like to first realize the $\dot{\operatorname{I}}$n\"on\"u--Wigner-contracted Lie algebra~$\heis_3 \oplus \uni(1,\HH)$ in $\mathcal{W}^{\hbar}$, and then deform this realization to that of~$\uni(2,\HH)$. To that end, we work with the embedding of $\heis_3 \oplus \uni(1,\HH)$ in $\mathcal{W}^{\hbar}$ given by the Lie algebra homomorphism
\begin{equation}
	\hspace{0.5 cm}  \scalebox{0.97}{$\mathsf{I} \mapsto  i \frac{1}{\hbar} \, , \qquad \zeta_{\dot{+} \dot{+}} \mapsto - 2 i \frac{  a}{\sqrt{\hbar}}\, , \qquad \zeta_{\dot{-} \dot{-}} \mapsto 2 i \frac{  a^\dagger}{\sqrt{\hbar}}\, , \qquad  \pi^\alpha_{\dot{\alpha}} \mapsto \frac{a^\alpha_{\dot{\alpha}}}{\sqrt{\hbar}} $} \,, \label{canonicalheis} 
	\end{equation} 
	\begin{equation}
	\scalebox{0.905}{$ \jmath^{++} \mapsto \mathcal{J}^{++} := -2i a^+_{\dot{-}} a^+_{\dot{+}}  \, , \,\,
	\jmath^{+-} = \jmath^{-+} \mapsto  \mathcal{J}^{+-} = \mathcal{J}^{-+} := -i ( a^+_{\dot{+}} a^-_{\dot{-}} + a^-_{\dot{+}} a^+_{\dot{-}} )    \, ,  \, \, \jmath^{--} \mapsto \mathcal{J}^{--} := -2i a^-_{\dot{-}} a^-_{\dot{+}}$}\,. \label{bilinearJ} 
\end{equation}
Equations~\eqref{canonicalheis} and \eqref{bilinearJ} give embeddings of $\heis_3$  and $\uni(1,\HH)$ into $\mathcal{W}^{\hbar}$, respectively. %For simplicity, we sometimes reuse the notation for the generators $a, a^\dagger, a^\alpha_{ \dot{\alpha}}$ of $\heis_3$ to label their images in $\mathcal{W}^{\hbar}$.
\smallskip

Since the bilinears in Equation~\eqref{bilinearJ} already obey $\uni(1,\HH)$, we do not deform them, and therefore ask whether we can find formal generators given by deformations
\begin{align}\label{ansatzK}
	\mathcal{K}_{\dot{+}\dot{+}} = -2i \frac{a}{\sqrt{\hbar}} + \cdots \, , \qquad \mathcal{K}_{\dot{+}\dot{-}} = \mathcal{K}_{\dot{-}\dot{+}} = i \frac{1}{\hbar} + \cdots\, , \qquad  \mathcal{K}_{\dot{-}\dot{-}} = 2i \frac{a^\dagger}{\sqrt{\hbar}} + \cdots \, ,
\end{align}
and
\begin{align}\label{ansatzP}
	\mathcal{P}^\alpha_{\dot{\alpha}} = \frac{a^\alpha_{\dot{\alpha}}}{\sqrt{\hbar}} + \cdots \, ,
\end{align}
in $\mathcal{W}^{\hbar}$ that realize $\uni(2,\HH)$. Here, the deformations denoted by $\cdots$'s are formal power series in $\sqrt{\hbar}$ taking values in $\mathcal{W}(\heis_3)$.  
\smallskip

\smallskip

\subsubsection{Gradings}
Certain gradings of the Lie algebra $\uni(2,\HH)$ will be particularly important for constructing the formal generators. 
Given a Lie algebra $\frg$ and a positive integer~$k$, we say that $D \in \frg$ gives a $|k|$-grading of~$\frg$ if we can write a direct sum decomposition $$ \frg = \frg_{-k} \oplus \cdots \oplus \frg_{-1} \oplus \frg_0 \oplus \frg_1 \oplus \cdots \oplus \frg_{k} $$
such that each $\frg_l$ above is an eigenspace of $D$ with eigenvalue $l$ under the adjoint action. We will use the same terminology for a Lie algebra $\frg \oplus \langle D \rangle$ even when the \textit{grading element}~$D \notin \frg$.
%Let $k>0$. Recall that a direct sum decomposition $$\frg = \bigoplus_{l=-k}^k \frg_l$$ of a Lie algebra $\frg$ into vector subspaces is called a \textit{$|k|$-grading} of $\frg$ if $[\frg_l,\frg_m] \subseteq \frg_{l+m}$, $\frg_{-k} \neq 0$, $\frg_{k} \neq 0$, and $\frg_{l} = 0$ for $|l|>k$. An element $D \in \frg$ such that $[D,X] = l X$ for all~$X \in \frg_l$ will be referred as the grading element and in that case, $l$ is called the grade of the element~$X$ (and of the subspace $\frg_l$) with respect to $D$. We shall say that $D$ gives a $|k|$-grading of~$\frg$. Note that $D$ must lie in the center of the Lie subalgebra~$\frg_0$~\cite[Chapter~3]{cap2009parabolic}. 

When performing the $\dot{\operatorname{I}}$n\"on\"u--Wigner contraction of Equation~\eqref{u2contraction}, we decomposed one of the direct summands~$\uni(1,\HH)$ in Equation \eqref{br343} into a direct sum of three 1-dimensional vector spaces using the $|1|$-grading
	\begin{align}
		\uni(1,\HH) = \langle K_{\dot{-}\dot{-}} \rangle  \oplus \uni(1,\CC) \oplus \langle K_{\dot{+}\dot{+}} \rangle  \label{1grading}	\, ,
	\end{align} 
with grading element $\frac{-iK_{\dot{+}\dot{-}}}{2}$. The entire Lie algebra has a $|2|$-grading
	\begin{align}
		\uni(2,\HH) = \langle K_{\dot{-}\dot{-}} \rangle \oplus \langle P^+_{\dot{-}}, P^-_{\dot{-}} \rangle  \oplus \underbrace{\left( \langle J^{\alpha \beta} \rangle \oplus \langle K_{\dot{+}\dot{-}} \rangle \right)}_{\uni(1,\HH) \oplus  \uni(1,\CC) }  \oplus  \langle P^+_{\dot{+}}, P^-_{\dot{+}} \rangle \oplus \langle K_{\dot{+}\dot{+}} \rangle  \label{3grading} 
	\, ,
	\end{align} 
with grading element $-iK_{\dot{+}\dot{-}}$.  
As $\lambda \to \infty$, the $\uni(1,\HH)$ of Equation~\eqref{1grading} contracts to
\begin{align*}
	\heis_1 = \langle \zeta_{\dot{-} \dot{-}} \rangle \oplus \langle \mathsf{I} \rangle \oplus \langle \zeta_{\dot{+} \dot{+}} \rangle \, .
\end{align*} 
In particular, the element $K_{\dot{+}\dot{-}}$ contracts to the central element $\mathsf{I}$. Recall that the entire Lie algebra $\uni(2,\HH)$ contracts to $\heis_3 \oplus \uni(1,\HH)$. Thus, when $\lambda \to \infty$, Equation \eqref{3grading} becomes 
\begin{align*}
	\uni(1,\HH) \oplus \heis_3 = \langle \zeta_{\dot{-} \dot{-}} \rangle \oplus \langle \pi^+_{\dot{-}}, \pi^-_{\dot{-}} \rangle  \oplus \underbrace{\left( \langle \jmath^{\alpha \beta} \rangle \oplus \langle \mathsf{I} \rangle \right)}_{\uni(1,\HH) \oplus \langle \mathsf{I} \rangle}  \oplus  \langle \pi^+_{\dot{+}}, \pi^-_{\dot{+}} \rangle \oplus  \langle  \zeta_{\dot{+} \dot{+}} \rangle \, .
\end{align*}

In $\mathcal{W}(\heis_3)$, the \textit{number operator} 
	\begin{align}\label{smalln}
		n := -\frac{1}{4}\xi_{\dot{\alpha} \dot{\beta}} \epsilon^{\dot{\alpha} \dot{\gamma}} \xi_{\dot{\gamma} \dot{\delta}}   \epsilon^{\dot{\delta} \dot{\beta}} - 1 = 2a^\dagger a \, .
	\end{align} 
Here, the $\eta$ trace-free tensor 
$\xi_{\dot{\alpha} \dot{\beta}} := \sqrt{\hbar} \, \zeta_{(\dot{\alpha} \dot{\beta})_\circ}$.
The operator $-n$ gives a $|2|$-grading of~$\heis_1$. Moreover, the \textit{total number operator} of the subalgebra $\heis_3$ of $\uni(1,\HH) \oplus \heis_3$ is given by
\begin{align}\label{bign}
	\mathcal{N} := n + N \, , \qquad \mbox{with} \quad 
		N := \frac{\hbar}{2} \, \pi^\alpha_{\dot{\alpha}} \epsilon^{\dot{\alpha} \dot{\beta}} \epsilon_{\alpha \beta} \pi^\beta_{\dot{\beta}} = a^+_{\dot{-}}a^-_{\dot{+}}-a^+_{\dot{+}}a^-_{\dot{-}}  \, \, .
\end{align}
The operator $-\mathcal{N}$ gives $\uni(1,\HH) \oplus \heis_3$ a $|2|$-grading. Moreover, $-N$ gives $\heis_2$ the $|1|$-grading $$ \heis_2 = \langle \pi^+_{\dot{-}}, \pi^-_{\dot{-}} \rangle  \oplus  \langle \mathsf{I} \rangle  \oplus  \langle \pi^+_{\dot{+}}, \pi^-_{\dot{+}} \rangle \, .$$ 

\subsubsection{Embedding}We would like to label particular elements of $\mathcal{W}^{\hbar}$ by meromorphic functions in $\sqrt{\hbar}$. More generally, a formal Laurent series $f^{\hbar}(x) := \sum_{k=M}^\infty c_k(x) \sqrt{\hbar}^k$ in~$\sqrt{\hbar}$, where $M \in \mathbb{Z}$ and $c_k(x)$ are polynomials in $x$, together with some $A \in \mathcal{W}(\heis_3)$ determines an element 
\begin{align}\label{meromorphicseries}
	f^{\hbar}(A) = \sum_{k=M}^\infty c_k(A) \sqrt{\hbar}^k \in \mathcal{W}^{\hbar}\, \, .
\end{align}
Elements of $\mathcal{W}^{\hbar}$ that can be expressed in the above form shall be called \textit{polymeromorphic}. For example, \begin{align}\label{squarerootlaurent}
	f^{\hbar}\left(N + \frac{n}{2}\right) = \sqrt{\frac{1}{\hbar} - N - \frac{n}{2}} \, \, ,
\end{align} denotes
\begin{align}\label{squarerootlaurentseries}
	 \frac{1}{\sqrt{\hbar}} \sum_{k=0}^{\infty} b_{k}\big(N + \tfrac{n}{2}\big) \, \hbar^k \, ,
\end{align} 
where
\begin{align*}
	 b_{k}(x)  := (-1)^k \frac{\prod_{l=0}^{k-1} \left(1 - 2l\right)}{2^k k!} x^k \, .
\end{align*}
Note that in this case $\sqrt{\frac{1}{\hbar} + x}$ is a \textit{bona fide} meromorphic function of $\sqrt{\hbar}$ around $\hbar = 0$. The polymeromorphic notion generalizes to formal Laurent series $f^{\hbar}(x,y,\dots, z)$ of many variables in the obvious way.

Multiplying a polymeromorphic element by any element in $\mathcal{W}^{\hbar}$ is well-defined. In particular, we have the following product properties.
\begin{lemma}\label{passagerules}
	Let $f^{\hbar}(n,N) \in \mathcal{W}^{\hbar}$ be polymeromorphic. Then,
	\begin{align*}
		f^{\hbar}(n,N) \hh a \, &= a f^{\hbar}(n-2,N) \, , \qquad \, \, \, \, \hh f^{\hbar}(n,N) a^\dagger = a^\dagger f^{\hbar}(n+2,N) \, , \\ 
		f^{\hbar}(n,N) a^+_{ \dot{+}} &= a^+_{ \dot{+}} f^{\hbar}(n,N-1) \, , \qquad f^{\hbar}(n,N) a^+_{ \dot{-}} = a^+_{ \dot{-}} f^{\hbar}(n,N+1) \, , \\ 
		f^{\hbar}(n,N) a^-_{ \dot{+}} &= a^-_{ \dot{+}} f^{\hbar}(n,N-1) \, , \qquad f^{\hbar}(n,N) a^-_{ \dot{-}} = a^-_{ \dot{-}} f^{\hbar}(n,N+1)  \, . 
	\end{align*}
\end{lemma}
\begin{proof}
	We prove the first of these properties, the rest follow \textit{mutatis mutandis}. By definition,
	$$f^{\hbar}(n,N) = \sum_{k=M}^\infty c_k(n,N) \sqrt{\hbar}^k \, \, ,$$
	for some integer $M$ and some polynomial $c_k(n,N)$ in $n$ and $N$. Therefore, we can express it unambiguously in the form $$c_k(n,N) = \sum_{j=0}^{m_1} \sum_{l=0}^{m_2} b_{jl} n^j N^l \, \, , \quad b_{jl} \in \CC \, ,$$
	since $n$ and $N$ commute. Note that
	$$ n^j a = a (n-2)^j \, .$$
	Thus, $$c_k(n, N) a = a \hh c_k(n-2,N)$$ and hence
	$$f^{\hbar}(n,N) a = a f^{\hbar}(n-2,N) \, .$$ 
\end{proof}
We will need to impose reality conditions on polymeromorphic elements. We call $f^{\hbar} \in \mathcal{W}^{\hbar}$ \textit{real polymeromorphic} if 
$$ (f^{\hbar})^\dagger = f^{\hbar} \, .$$ 

\smallskip
Now we can realize $\uni(2,\HH)$ in $\mathcal{W}^{\hbar}$. Let $\frg$ and $\mathfrak{k}$ be Lie algebras equipped with respective conjugate linear involutive automorphisms $\sigma_1$ and $\sigma_2$. We say that an embedding $\phi$ of~$\frg$ into $\mathfrak{k}$ is \textit{reality preserving} if the embedding intertwines $\sigma_1$ and $\sigma_2$: 
$$ \phi \circ \sigma_1 = \sigma_2 \circ \phi \, .$$
A reality preserving embedding maps fixed points of $\sigma_1$ to fixed points of $\sigma_2$, or in other words preserves real forms. The embeddings given in Equations \eqref{canonicalheis} and \eqref{bilinearJ} are reality preserving.
\medskip

%{\color{blue}  We have \begin{eqnarray*}
%		&\mathcal{K}_{\dot{+}\dot{+}} =  \sqrt{\frac{1}{\hbar} + \mathcal{N} - \frac{n}{2}} \; a \,\, , \qquad \mathcal{K}_{\dot{+}\dot{-}} = i \left( \frac{1}{\hbar} + \mathcal{N} \right) = \mathcal{K}_{\dot{-}\dot{+}} \,\, , \qquad\mathcal{K}_{\dot{-}\dot{-}} =   a^\dagger \sqrt{\frac{1}{\hbar} + \mathcal{N} - \frac{n}{2}} \,\, , &\\[1.5 mm]
%		&\mathcal{P}^+_{\dot{+}} = - c_2\, i \frac{1}{2} a^+_{\dot{-}}  \,  a   +  c_1 \sqrt{\frac{1}{\hbar} + \mathcal{N} - \frac{n}{2} } \; a^+_{\dot{+}}\,\, , \qquad \mathcal{P}^-_{\dot{-}} = - c_1 \, i \frac{1}{2} a^\dagger \, a^-_{\dot{+}}   + c_2 \, a^-_{\dot{-}} \sqrt{\frac{1}{\hbar} + \mathcal{N} - \frac{n}{2}}  \,\, ,& \\[1.5 mm]
%		&\mathcal{P}^+_{\dot{-}} =  - c_1 \, i \frac{1}{2} a^\dagger \, a^+_{\dot{+}}   + c_2\, a^+_{\dot{-}} \sqrt{\frac{1}{\hbar} + \mathcal{N} - \frac{n}{2}}\,\, , \qquad \mathcal{P}^-_{\dot{+}} =  - c_2\, i\frac{1}{2}  a^-_{\dot{-}} \, a    +  c_1 \sqrt{\frac{1}{\hbar} + \mathcal{N} - \frac{n}{2}} \; a^-_{\dot{+}}\,\, ,& \\[1.5 mm]
%		& \mathcal{J}^{++} = -2i a^+_{\dot{-}} a^+_{\dot{+}}  \,\, , \qquad
%		\mathcal{J}^{+-} = -i ( a^+_{\dot{+}} a^-_{\dot{-}} + a^-_{\dot{+}} a^+_{\dot{-}} ) = \mathcal{J}^{-+}  \,\, , \qquad  \mathcal{J}^{--} = -2i a^-_{\dot{+}} a^-_{\dot{-}}  \,\, .&
%\end{eqnarray*} This gives $\uni(2,\HH)$ for $c_1 c_2 = 1$ and gives $\uni(1,1,\HH)$ for $c_1 c_2 = -1$. Another way to get $\uni(1,1,\HH)$ is take the solution of \ref{firstembedding} and multiply $\mathcal{P}$s by an overall $i$.}

\begin{theorem}\label{firstembedding}
	A reality preserving embedding of the Lie algebra $\uni(2,\HH)$ into $\mathcal{W}^{\hbar}$ is given by \begin{align*}
		J^{\alpha\beta} \mapsto	\mathcal{J}^{\alpha\beta} \, , \qquad  P^\alpha_{\dot{\alpha}} \mapsto \mathcal{P}^\alpha_{\dot{\alpha}} \, , \qquad K_{\dot{\alpha}\dot{\beta}} \mapsto \mathcal{K}_{\dot{\alpha}\dot{\beta}} \,\, ,
	\end{align*} 
	where
	\begin{eqnarray*}
		&\mathcal{K}_{\dot{+}\dot{+}} = - 2i \sqrt{\frac{1}{\hbar} - N - \frac{n}{2}} \; a \,\, , \qquad \mathcal{K}_{\dot{+}\dot{-}} = i \left( \frac{1}{\hbar} - N - n \right) = \mathcal{K}_{\dot{-}\dot{+}} \,\, , \qquad\mathcal{K}_{\dot{-}\dot{-}} =  2i a^{\dagger} \sqrt{\frac{1}{\hbar} - N - \frac{n}{2}} \,\, , &\\[1.5 mm]
		&\mathcal{P}^+_{\dot{+}} = - a^+_{\dot{-}}  \,  a   +  \sqrt{\frac{1}{\hbar} - N - \frac{n}{2}} \; a^+_{\dot{+}}\,\, , \qquad \mathcal{P}^-_{\dot{-}} =  a^\dagger \, a^-_{\dot{+}}   + a^-_{\dot{-}} \sqrt{\frac{1}{\hbar} - N - \frac{n}{2}}  \,\, ,& \\[1.5 mm]
		&\mathcal{P}^+_{\dot{-}} =  a^\dagger \, a^+_{\dot{+}}   + a^+_{\dot{-}} \sqrt{\frac{1}{\hbar} - N - \frac{n}{2}}\,\, , \qquad \mathcal{P}^-_{\dot{+}} =  - a^-_{\dot{-}} \, a    +  \sqrt{\frac{1}{\hbar} - N - \frac{n}{2}} \; a^-_{\dot{+}}\,\, ,& \\[1.5 mm]
		& \mathcal{J}^{++} = -2i a^+_{\dot{-}} a^+_{\dot{+}}  \,\, , \qquad
		 \mathcal{J}^{+-} = -i ( a^+_{\dot{+}} a^-_{\dot{-}} + a^-_{\dot{+}} a^+_{\dot{-}} ) = \mathcal{J}^{-+}  \,\, , \qquad  \mathcal{J}^{--} = -2i a^-_{\dot{+}} a^-_{\dot{-}}  \,\, .&  
	\end{eqnarray*} 
\end{theorem}
\begin{proof}
	It suffices to verify that the generators given in the theorem satisfy the commutation relations
	\begin{eqnarray*}
		{}&	\scalebox{0.902}{$[\mathcal{J}^{\alpha \beta},\mathcal{J}^{\gamma \delta}] = i( \epsilon^{\alpha \gamma}\mathcal{J}^{\beta\delta}+\epsilon^{\alpha\delta}\mathcal{J}^{\beta\gamma}+\epsilon^{\beta \gamma}\mathcal{J}^{\alpha\delta}+\epsilon^{\beta \delta} \mathcal{J}^{\alpha \gamma})$} ,&\\[1mm]
		{}&	 \hspace{-0.19 cm}\scalebox{0.902}{$[\mathcal{J}^{\alpha \beta},\mathcal{P}^\gamma_{\dot{\gamma}}] = i( \epsilon^{\alpha \gamma}\mathcal{P}^\beta_{\dot{\gamma}}+\epsilon^{\beta \gamma}\mathcal{P}^\alpha_{\dot{\gamma}}), 
			\;\; 
			{}	 [\mathcal{P}^\alpha_{\dot{\alpha}},\mathcal{P}^\beta_{\dot{\beta}}] = i(  \epsilon_{\dot{\alpha} \dot{\beta}}\mathcal{J}^{\alpha \beta}+\epsilon^{\alpha \beta}\mathcal{K}_{\dot{\alpha}\dot{\beta}})\, ,
			\;\;
			{}[\mathcal{K}_{\dot{\alpha} \dot{\beta}},\mathcal{P}^\gamma_{\dot{\gamma}}] = i( \epsilon_{\dot{\alpha} \dot{\gamma}}\mathcal{P}^\gamma_{\dot{\beta}}+\epsilon_{\dot{\beta} \dot{\gamma}}\mathcal{P}^\gamma_{\dot{\alpha}})$} \,,
		 &\\[1mm]
		{}	&\scalebox{0.902}{$[\mathcal{K}_{\dot{\alpha} \dot{\beta}},\mathcal{K}_{\dot{\gamma} \dot{\delta}}] =  i(\epsilon_{\dot{\alpha} \dot{\gamma}}\mathcal{K}_{\dot{\beta}\dot{\delta}}+\epsilon_{\dot{\alpha}\dot{\delta}}\mathcal{K}_{\dot{\beta}\dot{\gamma}}+\epsilon_{\dot{\beta} \dot{\gamma}}\mathcal{K}_{\dot{\alpha} \dot{\delta}}+\epsilon_{\dot{\beta} \dot{\delta}} \mathcal{K}_{\dot{\alpha} \dot{\gamma}})$}\, .&
	\end{eqnarray*}
	In principle one can perform this verification directly, but this is rather tedious. We therefore outline a series of \textit{ans\"atze} built from polymeromorphic elements that produce the above embedding.
	
	We are tasked with finding the $\cdots$'s in Equations~\eqref{ansatzK} and~\eqref{ansatzP}. First, notice that the $|2|$-grading element $-iK_{\dot{+}\dot{-}}$ of $\uni(2,\HH)$ is mapped to $-i\mathcal{K}_{\dot{+}\dot{-}}$ in $\mathcal{W}^{\hbar}$. Moreover, recall that the operator $-\mathcal{N}$ gives a $|2|$-grading of $\uni(1,\HH) \oplus \heis_3$.  Therefore, we begin by adding~$-i\mathcal{N}$ to the central element $\frac{i}{\hbar}$ and consider the new ansatz $$ \mathcal{K}_{\dot{+}\dot{-}} = i \left( \frac{1}{\hbar} - \mathcal{N} \right) = \mathcal{K}_{\dot{-}\dot{+}} \,  \, .$$
	This will restrict the remaining generators considerably.
	
	Note that, with respect to the grading element $-i\mathcal{K}_{\dot{+}\dot{-}}$, the undeformed parts of $\mathcal{K}_{\dot{+}\dot{+}}$ and $\mathcal{K}_{\dot{-}\dot{-}}$ given by the images of $\zeta_{\dot{-}\dot{-}}$ and~$\zeta_{\dot{+}\dot{+}}$ (see Equation~\eqref{canonicalheis}) have grades $-2$ and $2$, respectively. Similarly, with respect to $-\mathcal{N}$, $a^\dagger$ and $a$ have respective grades $-2$ and $2$, and commute with~$\mathcal{J}^{\alpha\beta}$. Therefore, we make the further ansatz
	\begin{align*}
		\mathcal{K}_{\dot{+}\dot{+}} = - 2 i  F^{\hbar}(n,N) \, a \,\, , \qquad\mathcal{K}_{\dot{-}\dot{-}} =  2 i a^\dagger \, F^{\hbar}(n,N)  \,\, ,
	\end{align*}
	where we require $F^{\hbar}$ to be \textit{real} polymeromorphic in order to satisfy $(\mathcal{K}_{\dot{+}\dot{+}})^\dagger = \mathcal{K}_{\dot{-}\dot{-}}$.
	
	The generators $\mathcal{P}^\alpha_{\dot{\alpha}}$ must have grade $-1$ or $1$, with respect to $-i\mathcal{K}_{\dot{+}\dot{-}}$, and we have the following gradings:
	\begin{align*}
		\text{Grade -1: } \, a^-_{\dot{-}} \, , \, a^+_{\dot{-}}\, , \, a^\dagger a^-_{\dot{+}} \, , \, a^\dagger a^+_{\dot{+}} \, , \qquad \qquad  \text{Grade 1: } \, a^+_{\dot{+}} \, , \, a^-_{\dot{+}}\, ,\, a a^+_{\dot{-}} \, , \, a a^-_{\dot{-}} \, .
	\end{align*}
	In addition, the above combinations can be organized into doublet representations of~$\mathcal{J}^{\alpha\beta}$ and $\mathcal{K}_{\dot{\alpha}\dot{\beta}}$. This leads to our next ansatz.
\begin{eqnarray*}
	&\mathcal{P}^+_{\dot{+}} = -  a^+_{\dot{-}} \, H^{\hbar}_2(n,N) \, a    +  H^{\hbar}_1(n,N) \, a^+_{\dot{+}}\,\, , \qquad \mathcal{P}^-_{\dot{-}} =  a^\dagger \, H^{\hbar}_2(n,N) \, a^-_{\dot{+}}  +  a^-_{\dot{-}} \, H^{\hbar}_1(n,N) \,\, ,& \\[1.5 mm]
	&\mathcal{P}^+_{\dot{-}} =  a^\dagger \,  \, G^{\hbar}_2(n,N) \, a^+_{\dot{+}}  + a^+_{\dot{-}} \, G^{\hbar}_1(n,N)\,\, , \qquad \mathcal{P}^-_{\dot{+}} =  -  a^-_{\dot{-}}  \, G^{\hbar}_2(n,N) \, a  +  G^{\hbar}_1(n,N) \, a^-_{\dot{+}}\,\, .& 
\end{eqnarray*} 
	For the reality conditions $$(\mathcal{P}^+_{\dot{+}})^\dagger =  - \mathcal{P}^-_{\dot{-}} \, , \qquad (\mathcal{P}^+_{\dot{-}})^\dagger =   \mathcal{P}^-_{\dot{+}} \, ,$$
	it suffices to require that $G^{\hbar}_1,G^{\hbar}_2,H^{\hbar}_1,H^{\hbar}_2$ are \textit{real} polymeromorphic. Motivated by the classical solution~\eqref{classicalbracketsolution} with phases set to unity, we refine our ansatz to satisfy $$H_2^{\hbar} = 1 \, .$$ 
	We can now compute the Lie brackets of our \textit{ans\"atze} to obtain relations between the polymeromorphic elements $F^{\hbar}, G^{\hbar}_1, G^{\hbar}_2, H^{\hbar}_1$. 
	%We list a minimal\edz{check} set from which all remaining relations follow. 
	%To start with, from the brackets of \edz{ Fix the factors and signs in the next of this proof.} $\mathcal{K}_{\dot{\alpha}\dot{\beta}}$, we find 
	%$$ F^{\hbar}(n,N)^2 + \frac{n}{2} \left(F^{\hbar}(n,N)^2 - F^{\hbar}(n-2,N)^2 \right) = 4\left(\frac{1}{\hbar} - \mathcal{N} \right)\, .$$
	%This was obtained as follows. We compute \begin{align*}
	%	4\left(\frac{1}{\hbar} - \mathcal{N}\right) = -4i \mathcal{K}_{\dot{+}\dot{-}} = [\mathcal{K}_{\dot{+}\dot{+}},\mathcal{K}_{\dot{-}\dot{-}}] &= F^{\hbar}(n,N) a a^\dagger F^{\hbar}(n,N) - a^\dagger F^{\hbar}(n,N) F^{\hbar}(n,N) a\\
	%	&= F^{\hbar}(n,N)^2 + \frac{n}{2} F^{\hbar}(n,N)^2 - \frac{n}{2} F^{\hbar}(n-2,N)^2 \, ,
	%\end{align*} 
	%where while passing to the second equality we used Lemma~\ref{passagerules} and the identity $a a^\dagger = a^\dagger a + 1$. The rest of the relations follow similarly by computing the brackets and moving $a,a^\dagger,a^\alpha_{ \dot{\alpha}}$ to the left and functions of $n,N$ to the right via \ref{passagerules}.
	 Since $a,a^\dagger,a^\alpha_{ \dot{\alpha}}$ are independent, we can read off relations from their coefficients. The brackets of $\mathcal{J}^{\alpha\beta}$ with $\mathcal{P}^\alpha_{ \dot{\alpha}}$ imply that
	$$ G_1^{\hbar} = H_1^{\hbar} \, \, , \qquad G_2^{\hbar} = H_2^{\hbar} = 1  \, \, .$$ Then, examining independent combinations of $a,a^\dagger,a^\alpha_{ \dot{\alpha}}$ in the six possible brackets of~$\mathcal{P}^{\alpha}_{\dot{\alpha}}$ with $\mathcal{P}^{\beta}_{\dot{\beta}}$, we find
	\begin{eqnarray*}
		&H_1^{\hbar}(n,N)^2 = \frac{1}{\hbar} - \frac{n}{2} - N \, , &\\[1mm]
		%H_1^{\hbar}(n,N-1)  &= H_1^{\hbar}(n+2,N-2)  \, ,\\
		 %\left(1+\frac{n}{2}\right) &= \frac{n}{2}  + (H^{\hbar}_1(n,N))^2 - (H^{\hbar}_1(n,N-1))^2 + 2 \, , \\
		% H^{\hbar}_1(n,N)  &=  H_1^{\hbar}(n-2,N+1) \, ,\\
		% 2 &= H_1^{\hbar}(n,N-1)^2 - H_1^{\hbar}(n,N)^2 - \frac{n}{2}  + \left( 1 + \frac{n}{2} \right) \, , \\
		 %\frac{1}{\hbar} - 2 N - n + 1 &= (1-N) H_1^{\hbar}(n, N-1)^2 + N H_1^{\hbar}(n,N)^2 - \frac{n}{2}  \, , \\
		 &(1-N) H_1^{\hbar}(n+2,N-1) + (1+N) H_1^{\hbar}(n,N) =  2 \hh F^{\hbar}(n,N) \, .&
		 %F(n-2,N) &= (1-N) H_1^{\hbar}(n,N-1) H_2^{\hbar}(n-2,N-1) + (1+N) H_1^{\hbar}(n-2,N)H_2^{\hbar}(n-2,N) \, , \\
	\end{eqnarray*}
	%Lastly, the brackets of $\mathcal{K}_{\dot{\alpha}\dot{\beta}}$ and $P^\alpha_{ \dot{\alpha}}$ gives
	%\begin{align*}
	%	F^{\hbar}(n,N)H_1^{\hbar}(n+2,N) &= F^{\hbar}(n,N+1) H_1^{\hbar}(n,N) \, , \\
	%	F^{\hbar}(n,N)H_2^{\hbar}(n+2,N-1) &= F^{\hbar}(n+2,N-1) H_2^{\hbar}(n,N-1) \, , \\
	%	2 \hh H_2^{\hbar}(n,N-1) &= F^{\hbar}(n,N-1)H_1(n,N-1) - F^{\hbar}(n,N)H_1^{\hbar}(n+2,N-1) \, , \\
	%	2 \hh H_1^{\hbar}(n,N) &= \left(1+\frac{n}{2}\right) F^{\hbar}(n,N)H_2^{\hbar}(n,N) - \frac{n}{2} F^{\hbar}(n-2,N+1)\, .
	%\end{align*}
	%In summary $$ H_2^{\hbar} = G_2^{\hbar} = 1 \; , \qquad  F^{\hbar} = G_1^{\hbar} = H_1^{\hbar} = \sqrt{\frac{1}{\hbar} - N - \frac{n}{2}} $$
	%indeed leads to a solution since they satisfy the equations above, obtained by the Lie brackets. 
	Setting $$H_1^{\hbar} = \sqrt{\frac{1}{\hbar} - N - \frac{n}{2}} \, ,$$ so that in turn $$F^{\hbar} = H_1^{\hbar}\, ,$$ gives the quoted solution. It is easy to check that all remaining brackets are satisfied.
\end{proof}

\begin{remark}
	The above theorem does not address uniqueness. Indeed, on the basis of reality alone, it is at least possible to allow certain phase factors. For example, let $\psi,\phi_1,\phi_2$ be polynomials in $n,N$. Then 
\begin{eqnarray*}
	&\mathcal{K}_{\dot{+}\dot{+}} = - 2i e^{ i \psi} F^{\hbar}(n,N) \, a \, , \quad \mathcal{K}_{\dot{+}\dot{-}} = i \left( \frac{1}{\hbar} - n - N \right) = \mathcal{K}_{\dot{-}\dot{+}} \, , \quad\mathcal{K}_{\dot{-}\dot{-}} = 2 i a^\dagger \, e^{ -i \psi} F^{\hbar}(n,N)  \,\, ,&\\[1 mm]
	&\scalebox{0.92}{$\mathcal{P}^+_{\dot{+}} = - e^{ -i \phi_2} H^{\hbar}_2(n,N) a^+_{\dot{-}} \hh a \, + \, e^{ i \phi_1}H^{\hbar}_1(n,N) a^+_{\dot{+}}\, , \quad \mathcal{P}^-_{\dot{-}} =  a^\dagger \, a^-_{\dot{+}} \, e^{i \phi_2} H^{\hbar}_2(n,N)  +  a^-_{\dot{-}} e^{ - i\phi_1}\, H^{\hbar}_1(n,N) \,\, ,$}& \\[1.5 mm]
	&\scalebox{0.92}{$\mathcal{P}^+_{\dot{-}} =  a^\dagger \, a^+_{\dot{+}} \, e^{ i \varphi_2} G^{\hbar}_2(n,N)  + a^+_{\dot{-}} \, e^{ i \varphi_1}G^{\hbar}_1(n,N)\, , \quad \mathcal{P}^-_{\dot{+}} =  - e^{ -i \varphi_2}G^{\hbar}_2(n,N) a^-_{\dot{-}} \hh a +  e^{ -i \varphi_1}G^{\hbar}_1(n,N) \; a^-_{\dot{+}}\,\, ,$}&
\end{eqnarray*} 
also obey the reality conditions \eqref{spinorreality}. These phase factors were set to unity in the solution given in the theorem above, but were incorporated into the classical solution given in Subsection~\ref{sec:poissonalgebraembedding}. \hfill $\blacklozenge$
\end{remark}

Finally, we would like to emphasize that while the square root expressions appearing in Theorem~\ref{firstembedding} can be used to define operators on an appropriate Fock space, they in general do not have the desired antihermiticity properties required to construct a connection of a quantum dynamical system.
	\section{Beyond formal embeddings}\label{beyondformality}
	Our proof of the beyond formality Theorem~\ref{wehavesolutions2} will require us to promote the formal solution of Theorem~\ref{firstembedding} to operators acting on some Hilbert space. In particular, we wish to find unitary irreducible representations of the Lie algebra~$\uni(2,\HH)$ on finite-dimensional Hilbert spaces arising as truncations of the infinite-dimensional Hilbert space~$L^2(\RR^3)$. This generalizes the Holstein--Primakoff mechanism realizing finite-dimensional spin representations by truncations of the infinite-dimensional harmonic oscillator Fock space~\cite{holstein}. 
	\smallskip

	Partial sums of the formal power series appearing in the previous section are Laurent polynomials in $\sqrt{\hbar}$ with coefficients in $\mathcal{W}(\heis_3)$. To represent them as operators, we shall consider representations of the Heisenberg algebra. For that, let $\mathcal{S}$ be the dense subspace of Schwartz functions in $\mathcal{H} = L^2(\RR^3)$ and denote by $\gl(\mathcal{H})$ the set of all linear operators on $\mathcal{H}$. Note that $\gl(\mathcal{H})$ does not have an algebra structure. In order to construct an algebra, consider (see \cite[Chapter~2]{schmüdgen1990} and \cite[Chapter~3]{schmüdgen2020})  $$\widetilde{\gl}(\mathcal{S}) := \{A \in \gl(\mathcal{H}) \, \, | \, \, \mathcal{D}(A) = \mathcal{S}\, , \,\, A \hh \mathcal{S} \subseteq \mathcal{S}\,,\,\, \mathcal{S}\subseteq \mathcal{D}(A^*)\,,\,\, A^* \mathcal{S} \subseteq \mathcal{S}\},$$ where $\mathcal{D}$ denotes the domain of its argument and $A^*$ denotes the adjoint operator of~$A$ on $\mathcal{H}$. Set~$A^\dagger$ to be the restriction of this adjoint to $\mathcal{S}$. Then there is a Lie algebra homomorphism $\rho : \heis_3 \to \widetilde{\gl}(\mathcal{S})$, given by
	\begin{align*}
		\rho(a) &:=   \frac{1}{\sqrt{2}} \left(x_1 + \frac{\partial}{\partial x_1}\right)\, , \,\,\,\,\;\; \qquad \rho(a^\dagger)  \, :=  \frac{1}{\sqrt{2}} \left(x_1 - \frac{\partial}{\partial x_1}\right) \, , \;\;  \\
		\rho(a^+_{\dot{+}})  &:= - \frac{1}{\sqrt{2}} \left(x_2 + \frac{\partial}{\partial x_2}\right) \, , \;\; \qquad \rho(a^-_{\dot{-}})  := \frac{1}{\sqrt{2}} \left(x_2 - \frac{\partial}{\partial x_2}\right) \, , \;\;  \\
		\rho(a^-_{\dot{+}})  &:= \frac{1}{\sqrt{2}} \left(x_3 + \frac{\partial}{\partial x_3}\right) \, , \,\,\,\,\;\; \qquad \rho(a^+_{\dot{-}})  :=  \frac{1}{\sqrt{2}} \left(x_3 - \frac{\partial}{\partial x_3}\right) \, ,
	\end{align*} 
	and $\rho(1) := \Id $ (the identity operator). Here, $\big\{x_j, -i\frac{\partial}{\partial x_j}  \, \, | \, \, j = 1,2,3\big\}$ are position and momentum operators, whose domain of definition we take to be $\mathcal{S}$. We also denote by~$\rho$ the unital algebra homomorphism from $\mathcal{W}(\heis_3)$ to $\widetilde{\gl}(S)$ induced by the above Lie algebra representation of $\heis_3$ on $\mathcal{S}$. 
	\smallskip
	
	Recall that there is a countable complete orthonormal basis $\mathcal{B}$ of $\mathcal{H}$, given by (three-dimensional) Hermite functions in $\mathcal{S}$. We denote these by~$\ket{n_1,n_2,n_3}$, where $n_1, n_2, n_3$ are non-negative integers. The above operators act on these basis elements according to
	\begin{align*}
		\rho(a) \ket{n_1,n_2,n_3} &= \sqrt{n_1} \ket{n_1-1,n_2,n_3} \, , \,\,\,\,\;\;\;\; \rho(a^\dagger) \hh \ket{n_1,n_2,n_3} = \sqrt{n_1+1} \ket{n_1+1,n_2,n_3} \, , \;\;  \\
		\rho(a^+_{\dot{+}}) \ket{n_1,n_2,n_3} &= -\sqrt{n_2} \ket{n_1,n_2-1,n_3} \, , \;\;\;\; \rho(a^-_{\dot{-}}) \ket{n_1,n_2,n_3} = \sqrt{n_2+1} \ket{n_1,n_2+1,n_3} \, , \;\;  \\
		\rho(a^-_{\dot{+}}) \ket{n_1,n_2,n_3} &= \sqrt{n_3} \ket{n_1,n_2,n_3-1} \, , \,\,\,\,\;\;\;\; \rho(a^+_{\dot{-}}) \ket{n_1,n_2,n_3} = \sqrt{n_3+1} \ket{n_1,n_2,n_3+1} \, .  
	\end{align*} 
	Moreover,
	$$ \rho(n) \ket{n_1,n_2,n_3} = 2 \, n_1 \ket{n_1,n_2,n_3} \, , \qquad \hspace{0.39 cm}  \rho(N) \ket{n_1,n_2,n_3} = (n_2 + n_3 + 1) \ket{n_1,n_2,n_3}  \, .$$
	We would like to extend the domain of the map $\rho$ to polymeromorphic elements of Equation~\eqref{squarerootlaurent}. However, general formal Laurent series as given in Equation~\eqref{meromorphicseries} may not yield well-defined operators acting on Schwartz functions $\mathcal{S}$. This difficulty will be solved by suitably truncating $\mathcal{H}$, \textit{\`a la} Holstein and Primakoff.
	\smallskip
	
	Motivated by Equation~\eqref{squarerootlaurentseries}, we denote by $S_\ell(n,N,\hbar)$ the partial sum
	\begin{align*}
		S_\ell(n,N,\hbar) := \frac{1}{\sqrt{\hbar}} \sum_{k=0}^{\ell} b_k\left(N + \frac{n}{2}\right) \hbar^k \, ,
	\end{align*}
	so that we can now write 
	\begin{align}\label{truncatedsquarerootop}
		\scalebox{0.99}{$\rho(S_\ell(n,N,\hbar)) \ket{n_1,n_2,n_3} =  \frac{1}{\sqrt{\hbar}} \sum_{k=0}^\ell \frac{\prod_{l=0}^{k-1} (-1)^k \left(1 - 2l\right)}{2^k k!} \left(n_1+n_2+n_3+1\right)^k \hbar^k  \, \, \ket{n_1,n_2,n_3}$} \, .
	\end{align}
	Importantly, the formal parameter $\hbar$ can now stand for any positive real number. Assuming such,
	we may now study the convergence of the right hand side of the above equation for large $\ell$. (It will be clear from context when $\hbar$ is a number and when it is a formal parameter.) Let $m \in \mathbb{Z}_{>0}$ and
	\begin{align}\label{symmetrictensors}
		\mathcal{H}^{\frac{1}{m}} := \Span{\left\{ \,\, \ket{n_1,n_2,n_3} \in \mathcal{B} \,\,\,\,\, | \,\,\,\,\, n_1 + n_2 + n_3 \leq m - 1 \, \right\}} \; .
	\end{align} 
	This is a Hilbert space of finite dimension $m+2 \choose 3$. The operators $\rho\big(S_\ell(n,N,\frac{1}{m})\big)$ have well-defined actions on $\mathcal{H}^{\frac{1}{m}}$ for all $\ell \in \mathbb{N}$. Their action on basis elements is given in Equation~\eqref{truncatedsquarerootop}. Define the linear operator $\rho_S^{\sss\frac{1}{m}}: \mathcal{H}^{\frac{1}{m}} \to \mathcal{H}^{\frac{1}{m}}$ by
	\begin{align}\label{squarerootop}
		\rho_S^{\sss\frac{1}{m}} \ket{n_1,n_2,n_3} = \sqrt{m - n_1 - n_2 - n_3 - 1} \, \, \ket{n_1,n_2,n_3} \, .
	\end{align} 
	
	\begin{proposition}\label{squarerootconverges}
		Let $\hbar \in \mathscr{I} := \{\,\frac{1}{m} \, \,\, | \, \,\, m \in \ZZ_{>0} \, \}$. The sequence of operators $\rho\big(S_\ell(n,N,\hbar)\big)$ converges to the operator $\rho_S \in \gl(\mathcal{H}^{\hbar})$ in the strong operator topology.
	\end{proposition}
	\begin{proof}
		Convergence acting on states $\ket{n_1,n_2,n_3}$ such that $n_1 + n_2 + n_3 + 1 < m$ follows from the fact that $x = \hbar(n_1+n_2+n_3+1)$ is inside the radius of convergence of the power series for $\sqrt{1-x}$ around $x=0$. In fact, this power series converges absolutely when $x = 1$. This leads to the claimed limit. 
	\end{proof}
	Let us now consider the elements $(\mathcal{P}^\alpha_{ \dot{\alpha}})_\ell^{\hbar}$ and~$(\mathcal{K}_{\dot{\alpha}\dot{\beta}})_\ell^{\hbar}$ in $\mathcal{W}(\heis_3)$ obtained by replacing the Laurent series expansion of the square roots appearing in $\mathcal{P}^\alpha_{ \dot{\alpha}}$ and $\mathcal{K}_{\dot{\alpha}\dot{\beta}}$ with the partial sums $S_\ell(n,N,\hbar)$. For example, $$ (\mathcal{P}^+_{\dot{-}})_\ell^{\hbar} :=  a^\dagger \, a^+_{\dot{+}}   + a^+_{\dot{-}} \, S_\ell(n,N,\hbar) \, .$$
	This gives us a sequence of operators $\rho\big((\mathcal{P}^\alpha_{ \dot{\alpha}})_\ell^{\hbar}\big)$ and $\rho\big((\mathcal{K}_{\dot{\alpha}\dot{\beta}})_\ell^{\hbar}\big)$. Unlike the operators~$\rho\big(S_\ell(n,N,\hbar)\big)$ of Equation~\eqref{truncatedsquarerootop}, the image of these operators may not be contained in~$\mathcal{H}^{\frac{1}{m}}$, for example, 
	\begin{align*}
		\scalebox{0.945}{$\rho\big((\mathcal{P}^+_{\dot{-}})_\ell^{\sss\frac{1}{m}}\big) \ket{n_1,n_2,n_3}$} &=  \scalebox{0.945}{$\left(\rho(a^\dagger) \, \rho(a^+_{\dot{+}})   + \rho(a^+_{\dot{-}}) \rho\Big(S_\ell\big(n,N,\frac{1}{m}\big)\Big) \hh \right) \ket{n_1,n_2,n_3} $}\\
		&= \scalebox{0.945}{$-\sqrt{n_1+1}\sqrt{n_2}\ket{n_1+1,n_2-1,n_3}$} \\ & + \scalebox{0.945}{$\sqrt{m}\sqrt{n_3 + 1} \sum_{k=0}^\ell \frac{\prod_{l=0}^{k-1} (-1)^k \left(1 - 2l\right)}{2^k k!} \left(n_1+n_2+n_3+1\right)^k \frac{1}{m^k} \ket{n_1,n_2,n_3+1}$}\, .
	\end{align*}
	As seen from the above, successive applications of $\rho\big((\mathcal{P}^+_{\dot{-}})_\ell^{\sss\frac{1}{m}}\big)$ produces terms with ever increasing eigenvalue~$n_3$ of the third number operator $\rho(a^-_{\dot{+}})\rho(a^+_{\dot{-}})$. Therefore, we must treat $\rho\big((\mathcal{P}^\alpha_{ \dot{\alpha}})_\ell^{\sss\frac{1}{m}}\big)$ and $\rho\big((\mathcal{K}_{\dot{\alpha}\dot{\beta}})_\ell^{\sss\frac{1}{m}}\big)$ as maps from $\mathcal{H}^{\frac{1}{m}}$ to $\mathcal{H}$ when analyzing their large $\ell$ limits. To that end, define operators mapping $\mathcal{H}^{\frac{1}{m}} \to \mathcal{H}$ by
	\begin{align}\label{pkoperators}
		\rho_{\mathcal{K}_{\dot{+}\dot{+}}} &:= -2 \hh i \hh \rho_S^{\sss\frac{1}{m}} \hh \rho(a) \, , \qquad \qquad \qquad  \,\,\,\,\,\,\,  \rho_{\mathcal{K}_{\dot{-}\dot{-}}} := 2 \hh i \hh \rho(a^\dagger) \hh \rho_S^{\sss\frac{1}{m}} \, , \nonumber \\
		\rho_{\mathcal{P}^+_{\dot{+}}} &:= -\rho(a^+_{\dot{-}})\rho(a)   + \rho_S^{\sss\frac{1}{m}} \hh \rho(a^+_{\dot{+}}) \, , \,\, \qquad \rho_{\mathcal{P}^-_{\dot{-}}} := \rho(a^\dagger) \rho(a^-_{\dot{+}})  + \rho(a^-_{\dot{-}}) \rho_S^{\sss\frac{1}{m}} \, , \\
		\rho_{\mathcal{P}^-_{\dot{+}}} &:= - \rho(a^-_{\dot{-}}) \rho(a)   + \rho_S^{\sss\frac{1}{m}} \hh \rho(a^-_{\dot{+}}) \, , \,\, \qquad  \rho_{\mathcal{P}^+_{\dot{-}}} := \rho(a^\dagger) \rho(a^+_{\dot{+}})   + \rho(a^+_{\dot{-}}) \rho_S^{\sss\frac{1}{m}} \, . \nonumber
	\end{align}
	Here, we view $\rho(a^\dagger)$, $\rho(a^-_{\dot{-}})$, and~$\rho(a^+_{\dot{-}})$ as maps from $\mathcal{H}^{\frac{1}{m}} \to \mathcal{H}^{\frac{1}{m+1}}$ and  $\rho(a)$, $\rho(a^+_{\dot{+}})$, and~$\rho(a^-_{\dot{+}})$ as endomorphisms of $\mathcal{H}^{\frac{1}{m}}$. For example, that the composition of $\rho(a^\dagger)$ with~$\rho_S^{\sss\frac{1}{m}}$ is a sequence of maps $\mathcal{H}^{\frac{1}{m}} \to \mathcal{H}^{\frac{1}{m}} \to \mathcal{H}^{\frac{1}{m+1}} \subset \mathcal{H}$, while that of $\rho_S^{\sss\frac{1}{m}}$ with~$\rho(a)$ sends $\mathcal{H}^{\frac{1}{m}} \to \mathcal{H}^{\frac{1}{m}} \to \mathcal{H}^{\frac{1}{m}} \subset \mathcal{H}$. The following proposition establishes that the six operators displayed above in fact define maps from $\mathcal{H}^{\frac{1}{m}}$ to $\mathcal{H}^{\frac{1}{m}}$.
	\begin{proposition}\label{operatorconvergence}
		Let $\hbar \in \mathscr{I}$. The sequences of operators $\rho\big((\mathcal{P}^\alpha_{ \dot{\alpha}})_\ell^{\hbar}\big)$ and~$\rho\big((\mathcal{K}_{\dot{\alpha}\dot{\beta}})_\ell^{\hbar}\big)$ converge strongly to endomorphisms $\rho_{\mathcal{P}^\alpha_{\dot{\alpha}}}$ and $\rho_{\mathcal{K}_{\dot{\alpha}\dot{\beta}}}$ of~$\mathcal{H}^{\hbar}$, respectively. Moreover, the map~$\rho^{\hbar}$ defined by
		\begin{align*}
			J^{\alpha\beta} \mapsto	\rho(\mathcal{J}^{\alpha\beta}) \, , \qquad  P^\alpha_{\dot{\alpha}} \mapsto \rho_{\mathcal{P}^\alpha_{\dot{\alpha}}} \, , \qquad K_{\dot{\alpha}\dot{\beta}} \mapsto \rho_{\mathcal{K}_{\dot{\alpha}\dot{\beta}}} \,\, 
		\end{align*}
		gives a representation of the Lie algebra $\uni(2,\HH)$ on $\mathcal{H}^{\hbar}$.
	\end{proposition}
	\begin{proof}
		First, note that the linear operators $\rho(\mathcal{J}^{\alpha\beta})$ act according to
		\begin{align}\label{Jaction}
			\scalebox{1}{$\rho (\mathcal{J}^{++})\ket{n_1,n_2,n_3}$}&=\scalebox{1}{$2i\sqrt{n_3+1}\sqrt{n_2}\ket{n_1,n_2-1,n_3+1}$}\,, \nonumber\\
			\scalebox{1}{$\rho (\mathcal{J}^{+-})\ket{n_1,n_2,n_3}$}&=\scalebox{1}{$-i(n_3-n_2)\ket{n_1,n_2,n_3}$}\,,\\
			\scalebox{1}{$\rho (\mathcal{J}^{--})\ket{n_1,n_2,n_3}$}&=\scalebox{1}{$-2i\sqrt{n_2+1}\sqrt{n_3}\ket{n_1,n_2+1,n_3-1}$}\nonumber\,.
		\end{align}
		In particular, the eigenvalue $n_1+n_2+n_3+1$ of the number operator $\rho\big( N + \frac{n}{2} \big)$ is preserved by $\rho(\mathcal{J}^{\alpha\beta})$.
		Thus, the operators $\rho(\mathcal{J}^{\alpha\beta})$ are endomorphisms of $\mathcal{H}^{\hbar}$. 
		
		We must now analyze $\mathcal{P}^\alpha_{\dot{\alpha}}$ and $\mathcal{K}_{\dot{\alpha}\dot{\beta}}$. Firstly, the sequences of operators $\rho\big((\mathcal{P}^\alpha_{ \dot{\alpha}})_\ell^{\hbar}\big)$ and~$\rho\big((\mathcal{K}_{\dot{\alpha}\dot{\beta}})_\ell^{\hbar}\big)$ converge strongly to the linear operators $\rho_{\mathcal{P}^\alpha_{\dot{\alpha}}}$ and $\rho_{\mathcal{K}_{\dot{\alpha}\dot{\beta}}}$, respectively, by virtue of Proposition~\ref{squarerootconverges}. It remains to establish that $\rho_{\mathcal{P}^\alpha_{\dot{\alpha}}}$ and $\rho_{\mathcal{K}_{\dot{\alpha}\dot{\beta}}}$ are endomorphisms of~$\mathcal{H}^{\hbar}$. It suffices to consider their actions on the basis elements $\ket{n_1,n_2,n_3}$ of~$\mathcal{H}^{\hbar}$. Simple computations show
		\begin{align*}
			\scalebox{0.91}{$\rho_{ \mathcal{K}_{\dot{+}\dot{+}}} \ket{n_1,n_2,n_3}$} &= \scalebox{0.9}{$-2i \sqrt{n_1} \sqrt{\frac{1}{\hbar} - n_1 - n_2 - n_3 } \ket{n_1-1,n_2,n_3}$} \, ,\\
			\scalebox{0.91}{$\rho_{ \mathcal{K}_{\dot{+}\dot{-}}} \ket{n_1,n_2,n_3}$} &= \scalebox{0.9}{$ i \left(\frac{1}{\hbar} - 2 n_1 - n_2 - n_3 - 1\right) \ket{n_1,n_2,n_3}$} \, ,\\
			\scalebox{0.91}{$\rho_ {\mathcal{K}_{\dot{-}\dot{-}}} \ket{n_1,n_2,n_3}$} &= \scalebox{0.9}{$2i \sqrt{n_1 + 1} \sqrt{\frac{1}{\hbar} - n_1 - n_2 - n_3 - 1} \ket{n_1+1,n_2,n_3}$} \, ,\\
			\scalebox{0.91}{$\rho_{\mathcal{P}^+_{\dot{+}}} \ket{n_1, n_2, n_3}$} &= \scalebox{0.91}{$-\sqrt{n_1}\sqrt{n_3+1}\ket{n_1-1,n_2,n_3+1} - \sqrt{n_2}\sqrt{\frac{1}{\hbar} - n_1 - n_2 - n_3 } \ket{n_1,n_2-1,n_3}$}\, , \\
			\scalebox{0.91}{$\rho_{\mathcal{P}^-_{\dot{-}}} \ket{n_1, n_2, n_3}$} &= \scalebox{0.91}{$\sqrt{n_1+1}\sqrt{n_3}\ket{n_1+1,n_2,n_3-1} + \sqrt{n_2+1}\sqrt{\frac{1}{\hbar} - n_1 - n_2 - n_3 -1} \ket{n_1,n_2+1,n_3}$}\, , \\
			\scalebox{0.91}{$\rho_{\mathcal{P}^-_{\dot{+}}} \ket{n_1, n_2, n_3}$} &= \scalebox{0.91}{$-\sqrt{n_1}\sqrt{n_2+1}\ket{n_1-1,n_2+1,n_3} + \sqrt{n_3}\sqrt{\frac{1}{\hbar} - n_1 - n_2 - n_3 } \ket{n_1,n_2,n_3-1}$}\, , \\
			\scalebox{0.91}{$\rho_{\mathcal{P}^+_{\dot{-}}} \ket{n_1, n_2, n_3}$} &= \scalebox{0.91}{$-\sqrt{n_1+1}\sqrt{n_2}\ket{n_1+1,n_2-1,n_3} + \sqrt{n_3 + 1}\sqrt{\frac{1}{\hbar} - n_1 - n_2 - n_3 - 1} \ket{n_1,n_2,n_3+1}$}\, . 
		\end{align*}
		Now, recall that the basis elements $\ket{n_1,n_2,n_3}$ of $\mathcal{H}^{\hbar}$ satisfy $n_1+n_2+n_3 + 1 \leq \frac{1}{\hbar}$. But, the above operators do not preserve the eigenvalue of the total number operator $\rho(N + \frac{n}{2})$. However, it is easy to check that whenever a vector $\ket{a,b,c}$ with $a+b+c+1=\frac{1}{\hbar}+1$ is produced in the above formul\ae, a corresponding square root coefficient vanishes.
		
		One may use the formul\ae\ of the previous two displays above to show that
		\begin{align*} \label{lierepu2}
			\rho^{\hbar}([X,Y]) \ket{n_1,n_2,n_3} = [\rho^{\hbar}(X),\rho^{\hbar}(Y)] \ket{n_1,n_2,n_3} \, ,
		\end{align*}
		for any $X,Y \in \uni(2,\HH)$ and basis vector $\ket{n_1,n_2,n_3} \in \mathcal{H}^{\hbar}$. This establishes the Lie algebra representation statement.
	\end{proof}		
	\begin{remark}\label{re:symmetrictensors}
		The representations $\rho^{\hbar}$ on $\mathcal{H}^{\hbar}$ of Proposition~\ref{operatorconvergence} are in fact unitary irreducible representations of  dimension ${\sss\frac{1}{\hbar}}+2 \choose 3$ with fundamental weights $(\frac{1}{\hbar}-1,0)$. We shall call these \textit{symmetric tensor representations} because $\mathcal{H}^{\hbar}$ is realized by the space of symmetric tensors of rank strictly less than $\frac{1}{\hbar}$ in dimension three; see Equation~\eqref{symmetrictensors}. \hfill $\blacklozenge$
	\end{remark}

	Let us call the embedding map of Theorem~\ref{firstembedding} ``$\operatorname{embed}$'' and the map given in Proposition~\ref{operatorconvergence} from the image of $\operatorname{embed}$ in $\mathcal{W}^{\hbar}$ to unitary endomorphisms of $\mathcal{H}^{\hbar}$ ``$\operatorname{shorten}$''. Then, the results of this section are summarized as follows. 
	\begin{theorem}\label{th:summarydiagram}
		Let $\frac{1}{m} \in \mathscr{I}$. The diagram
		\begin{equation*}
			\begin{tikzcd}[row sep = huge, column sep = huge]
				& \hspace{1.3 cm}\operatorname{embed}(\uni(2,\HH)) \arrow[dr,"\operatorname{shorten}"] \subset \mathcal{W}^{\hbar}  &  \\
				\uni(2,\HH) \arrow[ur,hook,"\operatorname{embed}"] \arrow[rr,"\rho^{\frac{1}{m}}"] & & \uni(\mathcal{H}^{\frac{1}{m}}) 
			\end{tikzcd}
		\end{equation*} 
		commutes.
	\end{theorem}
	
	\begin{remark}\label{filteredvectorspaces}
		The dimension of the subspaces $\mathcal{H}^{\hbar}$ grows in the ``classical'' $\hbar = \frac{1}{m} \to 0$ limit. In fact, there is a vector space filtration
		\begin{align*}
			\mathcal{H}^{1} \subset \mathcal{H}^{\frac{1}{2}} \subset \mathcal{H}^{\frac{1}{3}} \subset \dots \subset \mathcal{H} \, .
		\end{align*}
		A similar phenomenon was observed in~\cite{landsmancoadjoint}. \hfill $\blacklozenge$
	\end{remark}
	\section{Quantization}\label{sec:quantization}
	\subsection{Formal quantization}\label{sec:formalquantization} We are almost ready to prove Theorem~\ref{wehavesolutions2}, but still need to give a detailed construction of the Hilbert bundle. 
	The Hilbert space $\mathcal{H} = L^2(\RR^3)$ is a unitary representation of $\Mpc(6,\RR)$~\cite{robinson1989metaplectic}. The Lie group $\Mpc(2n,\RR)$ is constructed as follows~\cite{robinson1989metaplectic}. Because $\Uni(n,\CC)$ is the maximal compact subgroup of $\Sp(2n,\RR)$, we have that $2 \ZZ < \ZZ = \pi_1(\Sp(2n,\RR)) $. The metaplectic group $\Mp(2n,\RR)$ is the unique connected double cover of $\Sp(2n,\RR)$, corresponding to the short exact sequence of Lie groups
	$$\begin{tikzcd} 1 \arrow[r] & \ZZ_2 \arrow[r] & \Mp(2n,\RR) \arrow[r, "\sigma"] & \Sp(2n,\RR) \arrow[r] & 1 \end{tikzcd} \, \, .$$ 
	The preimage of the identity $\sigma^{-1}(e) \cong \mathbb{Z}_2$. Moreover $\Uni(1,\CC)$ also has a subgroup $\ZZ_2 \cong \{-1,1\}$, so the metaplectic-c group~$\Mpc(2n,\RR)$ is given by
	\begin{align*}
		\Mpc(2n,\RR) := \Mp(2n,\RR) \times_{\mathbb{Z}_2} \Uni(1,\CC) \, .
	\end{align*}
	These groups are related by the following commutative diagram~\cite{forgerhess1979,hess1981}
	\begin{equation}\label{mpcdiagram}
	\begin{tikzcd} & 1 \arrow[d] & 1 \arrow[d] & 1 \arrow[d] & \\
	1 \arrow[r] & \ZZ_2 \arrow[r] \arrow[d] & \Mp(2n,\RR) \arrow[r, "\sigma"]  \arrow[d]& \Sp(2n,\RR) \arrow[r] \arrow[d,equal] & 1 \\
	1 \arrow[r] & \Uni(1,\CC) \arrow[r] \arrow[d] & \Mpc(2n,\RR) \arrow[r, "\tau"] \arrow[d]& \Sp(2n,\RR) \arrow[r] \arrow[d] & 1 \\
	1 \arrow[r] & \Uni(1,\CC) \arrow[r,equal] \arrow[d] & \Uni(1,\CC) \arrow[r] \arrow[d] & 1 & \\
	& 1 & 1 & & \end{tikzcd}
	\end{equation}
	whose columns and rows are exact. Note that $\tau([A,\lambda])=\sigma(A)$ where $A \in \Mp(2n,\RR)$.
	
	As mentioned in the introduction, the bundle of symplectic frames $\operatorname{Fr}_s(\xi)$ on the distribution $\xi$ is a principal $\Sp(6,\RR)$-bundle and admits a lift of the structure group to a principal~$\Mpc(6,\RR)$-bundle $\mathscr{P} \overset{\tilde{\pi}}{\longrightarrow} \operatorname{Fr}_s(\xi)$, where $\tilde{\pi}$ is principal bundle homomorphism~\cite{robinson1989metaplectic}. The associated Hilbert bundle 
	$$ \mathcal{H}Z := \mathscr{P}  \times_{\Mpc(6,\RR)} \mathcal{H} \, $$
	is known as the symplectic spinor bundle~\cite{symplecticdirac2}. 
	\subsubsection{Homogeneous model}\label{homogeneoushilbert}
We now want to relate the Hilbert bundle $\mathcal{H}Z$ to the homogeneous model. Note that generally, if $M\cong G/H$ is a homogeneous space, the frame bundle~$\operatorname{Fr}(M)$ is isomorphic to $G \times_H \GL(T_{[e]} M)$, where $[e]$ is the identity coset. Specializing to~$M$ the seven sphere with~$H = \Uni(1,\HH)$, one has an~$H$-invariant contact form~$\alpha$, and an~$H$-invariant two-form~$d\alpha$; see Section~\ref{sec:contactsphere}. Moreover the contact form $\alpha$ determines the distribution~$\xi$, which is preserved by~$\Uni(1,\HH)$. The restriction of $d\alpha$ to $\xi$ makes each fiber a symplectic vector space with a symplectic bilinear form preserved by~$\Sp(6,\RR)$. Thus, $\Uni(1,\HH)$ embeds into $\Sp(6,\RR)$. From this embedding and Diagram~\eqref{mpcdiagram} above, we can construct an embedding of~$\Uni(1,\HH)$ into~$\Mpc(6,\RR)$. To begin with, we lift the $\Sp(6,\RR)$ embedding of~$\Uni(1,\HH)$ to an~$\Mp(6,\RR)$ embedding 
	$$
		\begin{tikzcd}
			 &\Mp(6,\RR) \arrow[d]\\\Uni(1,\HH) \arrow[ur,hook,dashed,"s"]  \arrow[r,hook] & \Sp(6,\RR) \, .
		\end{tikzcd}
	$$
	%The lift $s$ exists since $\Uni(1,\HH)$ is simply-connected. The above diagram commutes, so $s$ is an injective immersion. As $\Uni(1,\HH)$ is compact, it follows that $s$ is an embedding. Moreover, uniqueness of the lift $s$ implies that $s$ is a Lie group homomorphism.
	This follows because $\Uni(1,\HH)$ is compact and simply-connected. From Diagram~\eqref{mpcdiagram}, $\Mp(6,\RR)$ is a closed subgroup of~$\Mpc(6,\RR)$, and hence is an embedded Lie subgroup of~$\Mpc(6,\RR)$. Composing $s$ with this embedding, we obtain an embedding of~$\Uni(1,\HH)$ into~$\Mpc(6,\RR)$.
	\smallskip
	
	Now, the principal $\Uni(1,\HH)$-bundle $\Uni(2,\HH) \to \Uni(2,\HH)/\Uni(1,\HH) \cong S^7$ extends both to a principal $\Sp(6,\RR)$-bundle $\Uni(2,\HH)\times_{\Uni(1,\HH)} \Sp(6,\RR)$, which is isomorphic to the symplectic frame bundle $\operatorname{Fr}_s(\xi)$, and also to a principal $\Mpc(6,\RR)$-bundle $\Uni(2,\HH)\times_{\Uni(1,\HH)} \Mpc(6,\RR)$. In fact, the latter principal bundle $\Uni(2,\HH)\times_{\Uni(1,\HH)} \Mpc(6,\RR)$ is an equivariant lift of the former bundle $\Uni(2,\HH)\times_{\Uni(1,\HH)} \Sp(6,\RR)$. To see this, we observe that the following diagram commutes
	\begin{equation}
		\begin{tikzcd}[row sep = 1.7 cm, column sep = tiny]
			\Uni(2,\HH)\times_{\Uni(1,\HH)} \Mpc(6,\RR) \arrow[rr,"\pi_\tau"] \arrow[rd,"\pi_m"]
			& & \Uni(2,\HH)\times_{\Uni(1,\HH)} \Sp(6,\RR) \arrow[ld,"\pi_s"]\\
			& S^7 \cong \Uni(2,\HH) / \Uni(1,\HH) \,.& 
		\end{tikzcd} 
	\end{equation}
	The map $\pi_\tau$ is defined by $\pi_\tau([g,\tilde{A}]) = [g,\tau(\tilde{A})]$. It is also equivariant with respect to the right actions of the structure groups on the principal bundles in the sense that $\pi_\tau(x \tilde{A}) = \pi_\tau(x) \tau(\tilde{A})$ for all $x \in \Uni(2,\HH)\times_{\Uni(1,\HH)} \Mpc(6,\RR)$ and~$\tilde{A} \in \Mpc(6,\RR)$. Now, because $\operatorname{Fr}_s(\xi) \cong \Uni(2,\HH)\times_{\Uni(1,\HH)} \Sp(6,\RR)$, we can identify $\Uni(2,\HH)\times_{\Uni(1,\HH)} \Mpc(6,\RR)$ as the lift of the structure group~$\Sp(6,\RR)$ of $\operatorname{Fr}_s(\xi)$. It is known that isomorphism classes of $\Mpc(2n,\RR)$-structures over some manifold $Z$ are classified by the first sheaf cohomology group $H^1\big(Z,\underline{\Uni(1,\CC)}\big) \cong H^2(Z,\ZZ)$~\cite{forgerhess1979,robinson1989metaplectic}, where $H^2(Z,\ZZ)$ is the second \v{C}ech cohomology group and is trivial for the case $Z = S^7$. Hence, the above lift is unique up to isomorphism of $\Mpc(6,\RR)$-structures. So, we now have the isomorphism $$\mathscr{P} \cong \Uni(2,\HH)\times_{\Uni(1,\HH)} \Mpc(6,\RR)$$ and a natural interpretation of the bundles associated to the metaplectic-c structure coming from the extended bundles of the homogeneous model and therefore, the homogeneous model itself. Namely,
	$$ \mathcal{H}Z \cong \mathscr{P} \times_{\Mpc(6,\RR)} \mathcal{H} \cong \big( \Uni(2,\HH)\times_{\Uni(1,\HH)} \Mpc(6,\RR) \big) \times_{\Mpc(6,\RR)} \mathcal{H} \cong \Uni(2,\HH)\times_{\Uni(1,\HH)} \mathcal{H} \, .$$
	We denote any one of these associated bundles by $\mathcal{H}Z$ and $\Uni(2,\HH)\times_{\Uni(1,\HH)} \Mpc(6,\RR)$ by~$\mathscr{P}$. We now return to the formal connection Problem~\ref{wehaveproblems2}, which will be solved for $S^7$ by Theorem~\ref{wehavesolutions2}. 
	
	Finally, for concreteness, we need to specify what we mean by the space of smooth sections $\Gamma(\mathcal{H}Z)$ of $\mathcal{H}Z$. A general discussion is given in~\cite{symplecticdirac}. For our purposes, we assume that sections obey a Schwartz property in the fiber and are smooth along the base. More precisely, in a local trivialization, $\Psi \in \Gamma(\mathcal{H}Z)$ is given by $\psi(z,x)$, where~$z\in \RR^7$ are base coordinates and $x \in \RR^3$ are fiber coordinates, while $\psi(z,x)$ is a Schwartz function of $x$ and a smooth function of $z$.
	
	\subsubsection{Formal connections}
	From the above discussion, one might think that to prove Theorem~\ref{wehavesolutions2} all that is needed is to induce a connection on $\mathcal{H}Z$ from the Maurer--Cartan connection on $\Uni(2,\HH)$. For that, one would need a unitary~$\Uni(2,\HH)$-representation acting on $\mathcal{H}$, but  Theorem~\ref{firstembedding} only gives formal operators. Those however will suffice to prove Theorem~\ref{wehavesolutions2}.
	
	\begin{proof}[Proof of Theorem 1.3]
		We need to write down a sequence of connections $\nabla^{\hbar,\ell}$ acting on smooth sections of $\mathcal{H}Z$, whose curvature equals $\hbar^{\ell/2} F_{\ell}$, where $F_\ell$ is some smooth two-form that is polynomial in $\sqrt{\hbar}$ and takes values in $\End(\mathcal{H}Z)$. We consider $\ell \geq 0$ because the cases $\ell = -2$, and $\ell = -1$ have already been handled; see Equations~\eqref{minustwonabla} and \eqref{minusonenabla}. Once again, we trivialize bundles by choosing the two patches described in Subsection~\ref{subsec:homogeneousmodel}. Let us consider
		$$ \nabla^{\hbar,\ell} = d + \mathcal{A}^{\hbar, \ell} \, , $$
		where 
		\begin{align}
			 \mathcal{A}^{\hbar, \ell} &= i \, \alpha \left(\frac{1}{\hbar} - N - n\right) -   2 i \, \mu_{+-} (a^+_{\dot{+}} a^-_{\dot{-}} + a^-_{\dot{+}} a^+_{\dot{-}})   
			 \\&+  \left( - 2 i \, \mu_{++} a^+_{\dot{-}} \,  a^{+}_{\dot{+}} - \nu^{\dot{+}}_{+} a^{+}_{\dot{-}} \, a  - \nu^{\dot{+}}_{-}\, a^{-}_{\dot{-}} \hh a + \frac{1}{\sqrt{\hbar}}   \sum_{k=0}^{\ell+1} b_{k} \, \hbar^k \left( - 2 i \hh \kappa^{\dot{+}\dot{+}}\, a  + \nu^{\dot{+}}_{+} a^{+}_{\dot{+}} + \nu^{\dot{+}}_{-}  a^{-}_{\dot{+}} \right) \right)  \nonumber 
			 \\&+ \left( - 2 i \, \mu_{--}   a^{-}_{\dot{+}} \, a^-_{\dot{-}} + \nu^{\dot{-}}_{-} a^\dagger a^{-}_{\dot{+}} + \nu^{\dot{-}}_{+}\, a^\dagger a^{+}_{\dot{+}} + \left(  2 i \hh \kappa^{\dot{-}\dot{-}}\, a^\dagger  + \nu^{\dot{-}}_{-} a^{-}_{\dot{-}} + \nu^{\dot{-}}_{+}  a^{+}_{\dot{-}} \right) \frac{1}{\sqrt{\hbar}}   \sum_{k=0}^{\ell+1} b_{k} \, \hbar^k  \right)  \nonumber .
		\end{align} 
	The above connection form $\mathcal{A}^{\hbar, \ell}$ was computed by coupling the generators~$(\mathcal{P}^\alpha_{\dot{\alpha}})_\ell^{\hbar}$,~$ (\mathcal{K}_{\dot{\alpha}\dot{\beta}})_\ell^{\hbar}$, and $ \mathcal{J}^{\alpha\beta}$, obtained from the formal embedding of Theorem~\ref{wehavesolutions2}, to the exterior differential system of Equation~\eqref{exteriorsystem} according to Equation~\eqref{ansatzconnection}.
	Recall that $\alpha = 2 \hh \kappa^{\dot{+}\dot{-}}$ and~$b_k$ were given by $$ b_k = (-1)^k\frac{\prod_{l=0}^{k-1} \left(1 - 2l\right)}{2^k k!} \left(N + \frac{n}{2}\right)^k \, . $$
	Note that the hermitean conjugate $\dagger$ can be extended to act on differential forms:
	$$ \alpha^\dagger = \alpha \, , \,\,\,\, (\kappa^{\dot{+} \dot{+}})^\dagger = - \kappa^{\dot{-} \dot{-}} \, , \,\,\,\, (\nu^{\dot{+}}_+)^\dagger = \nu^{\dot{-}}_- \, , \quad (\nu^{\dot{+}}_-)^\dagger = - \nu^{\dot{-}}_+ \, , \,\,\,\, (\mu_{++})^\dagger = - \mu_{--} \, , \,\,\,\, (\mu_{+-})^\dagger = \mu_{+-} \, ,$$
	see Equation~\eqref{doubleindexforms}. It then follows that $(\mathcal{A}^{\hbar, \ell})^\dagger = - \mathcal{A}^{\hbar, \ell} $.
	
	We observe that the connection $\nabla^{\hbar,-2}$ on $\mathcal{H}Z$ can be globally expressed as $d + \frac{i}{\hbar} \alpha \hh 1 $ upon trivializing. Moreover, $\nabla^{\hbar,\ell'} - \nabla^{\hbar,\ell}$ is a one-form valued in $\End(\mathcal{H}Z)$ for all $\ell, \ell'$ on both patches. On the overlap, the systems of differential forms on each patch are related by a $\Uni(1,\HH)$ gauge transformation. But $\Uni(1,\HH)$ acts unitarily on $\mathcal{H}$ because $\mathcal{H}$ is a unitary representation of $\Mpc(6,\RR) \supseteq \Uni(1,\HH)$. It follows that $\nabla^{\hbar, \ell}$ indeed defines a connection on $\mathcal{H}Z$.

	We claim that the sequence of connections $\nabla^{\hbar,\ell}$ solves Problem~\ref{wehaveproblems2} for the case $Z = S^7 \cong \Uni(2,\HH) / \Uni(1,\HH)$. Indeed, Requirement \eqref{formalhermitean} of Problem~\ref{wehaveproblems2} holds because $(\mathcal{A}^{\hbar, \ell})^\dagger = -\mathcal{A}^{\hbar, \ell}$ for every $\ell$. Also, observe that $\frac{1}{\hbar}$ only appears in the coefficient of $\alpha$. Thus, after multiplying by $-i\hbar$, the global contact one-form~$\alpha$ is the only term surviving in the $\hbar \to 0$ limit. This shows that Requirement \eqref{formallimit} holds for all $\ell$. Lastly, Theorem~\ref{firstembedding} and the exterior differential system of Equation~\eqref{exteriorsystem} together guarantee that, for any $\ell$, the curvature~$(\nabla^{\hbar,\ell})^2$ is flat up to order $\ell$ (meaning it equals $\hbar^{\ell/2} F_{\ell}$ for some smooth two-form $F_\ell$ polynomial in~$\sqrt{\hbar}$).
	\end{proof}
	
	\subsection{Beyond formal quantization}\label{sec:beyondformalconnection}
	
	For distinguished discrete values of $\hbar$, Theorem~\ref{th:beyondformality} shows that the formal connection gives a quantum connection of a corresponding quantum dynamical system. In other words, Theorem~\eqref{th:beyondformality} produces quantum mechanical systems from a formal quantization.  
	
	\begin{proof}[Proof of Theorem~\ref{th:beyondformality}]
	We start by constructing the claimed vector bundle filtration. For each subspace $\mathcal{H}^{\hbar}$ of $\mathcal{H}$ appearing in the vector space filtration of Remark~\ref{filteredvectorspaces}, we can form the corresponding vector bundle \begin{align*}
		\mathcal{H}^{\hbar} Z = \Uni(2,\HH) \times_{\Uni(1,\HH)} \mathcal{H}^{\hbar}
	\end{align*} 
	associated to the homogeneous model $\Uni(2,\HH)/\Uni(1,\HH)$. The metaplectic representation restricted to $\Uni(1,\HH)$ agrees with the representation of $\Uni(1,\HH)$ obtained by exponentiating the Lie algebra representation $\rho^{\hbar}$ restricted to $\uni(1,\HH)$. One thus obtains the claimed vector bundle filtration of Theorem~\ref{th:beyondformality}:
	\begin{align*}
		\mathcal{H}^1Z \subset \mathcal{H}^{1/2}Z \subset \mathcal{H}^{1/3}Z \subset \cdots \subset \mathcal{H}Z \, .
	\end{align*} 
	
	\smallskip
	
	In general, given a homogeneous space $G/H$ and a $G$-representation $\mathbb{V}$, the associated vector bundle $G \times_H \mathbb{V}$ is a (model) tractor bundle~\cite{cap2009parabolic}. The (flat) Cartan connection $\omega$ on~$G$ defines a flat linear connection $\nabla^{\omega}$ on $G \times_H \mathbb{V}$. When $\mathbb{V}$ is a unitary representation, the pair $(G \times_H \mathbb{V},\nabla^{\omega})$ is a quantum dynamical system. In particular, equipping~$\mathcal{H}^{\hbar}Z$ with the connection $\nabla^{\hbar}$ defined by the model $\Uni(2,\HH)$ Cartan connection gives a sequence finite-dimensional quantum dynamical systems $(\mathcal{H}^{\hbar}Z,\nabla^{\hbar})_{\hbar\in \mathscr{I}}$. 
	
	The proof of Theorem~\ref{th:beyondformality} is completed upon showing that these same quantum dynamical systems are induced by the formal connection $\nabla^{[\hbar]}$ of Theorem~\ref{wehavesolutions2} on $\mathcal{H}Z$. The limit statement of Equation~\eqref{limitconnections} deals with sections of a vector bundle. Our plan is to write explicit expressions for the connections $\nabla^{\hbar}$ and $\nabla^{[\hbar]}$ in the trivialization associated to the coordinate patches of Subsection~\ref{subsec:homogeneousmodel}.
	
	For any smooth section $\Psi \in \Gamma(\mathcal{H}^{\hbar}Z)$, there is a~$\Uni(2,\HH)$-equivariant function $\hat{\Psi}$ on the extended principal bundle $\widehat{\Uni(2,\HH)} : = \Uni(2,\HH) \times_{\Uni(1,\HH)} \Uni(2,\HH)$ with values in $\mathcal{H}^{\hbar}$. Then, $\nabla^{\hbar} \Psi$ corresponds to the horizontal, equivariant form $d \hat{\Psi} + \rho^{\hbar}(\hat{\omega}) \hat{\Psi}$, where $\hat{\omega}$ is the principal connection on $\widehat{\Uni(2,\HH)}$ coming from the Maurer--Cartan form $\omega$ on $\Uni(2,\HH)$ (see Subsection~\ref{subsec:homogeneousmodel}). It follows, in the trivializations corresponding to the patches, that $$\nabla^{\hbar} \Psi = d \hh \Psi + \rho^{\hbar}(A_s) \Psi \, ,$$ where $A_s$ is the pullback of $\omega$ along the section $s$. Thus, $\nabla^{\hbar}$ is represented as $$ \nabla^{\hbar} = d + A_s = d + \kappa^{\dot{\alpha}\dot{\beta}}\, K_{\dot{\alpha} \dot{\beta}} + \nu^{\dot{\alpha}}_\alpha \, P^\alpha_{ \dot{\alpha}} + \mu_{\alpha\beta} \, J^{\alpha\beta} \, . $$
	On the other hand, the formal connection $$\nabla^{\hbar,\ell} = d + \mathcal{A}^{\hbar,\ell} = d + \kappa^{\dot{\alpha}\dot{\beta}}\, (\mathcal{K}_{\dot{\alpha}\dot{\beta}})_\ell^{\hbar} + \nu^{\dot{\alpha}}_\alpha \, (\mathcal{P}^\alpha_{ \dot{\alpha}})_\ell^{\hbar} + \mu_{\alpha\beta} \, \mathcal{J}^{\alpha\beta} \, $$
	defines an operator mapping $\Gamma(TZ) \times \Gamma(\mathcal{H}^{\hbar}Z)$ to $\Gamma(\mathcal{H}Z)$ because the operators $\rho\big((\mathcal{K}_{\dot{\alpha}\dot{\beta}})_\ell^{\hbar} \big)$ and  $\rho\big((\mathcal{P}^\alpha_{ \dot{\alpha}})_\ell^{\hbar} \big)$ are maps from $\mathcal{H}^{\hbar}$ to $\mathcal{H}$. We now need to show for all $\Psi \in \Gamma(\mathcal{H}^{\hbar})$ and $u \in \Gamma(TZ)$ that $$\nabla^{\hbar}_u \hh \Psi = \lim_{\ell \to \infty} \nabla^{\hbar, \hh \ell}_u \hh \Psi \,\,  . $$
	The above limit holds pointwise by virtue of Proposition~\ref{operatorconvergence}.
	\end{proof}
	
	\section*{Acknowledgments}
	We thank Petr Vlachopulos for collaboration during the early stages of this work. S.C. thanks Motohico Mulase for discussions. S.C. and A.W. also thank Emanuele Latini for discussions. C.G. thanks Andrew Beckett, Jos\'e Figueroa-O'Farrill, and Dieter Van den Bleeken for many useful discussions on contact geometry and homogeneous spaces, and also thanks Roger Casals for his valuable comments and conversations. A.W. was supported by Simons Foundation Collaboration Grant for Mathematicians ID 686131. 
	
\bibliographystyle{utphys}
\bibliography{draft}

\providecommand{\href}[2]{#2}\begingroup\raggedright\begin{thebibliography}{10}

\bibitem{souriau1966}
J.-M. Souriau, ``Quantification g\'eom\'etrique,'' {\em Comm. Math. Phys.}
  {\bfseries 1} (1966) 374--398.
  \url{http://projecteuclid.org/euclid.cmp/1103758996}.

\bibitem{Kostant1970}
B.~Kostant, \href{http://dx.doi.org/10.1007/BFb0079068}{``Quantization and
  unitary representations. {I}. {P}requantization,''} in {\em Lectures in
  {M}odern {A}nalysis and {A}pplications {III}}, vol.~170 of {\em Lecture Notes
  in Mathematics}, pp.~87--208.
\newblock Springer, Berlin, Heidelberg, 1970.
\newblock \url{https://doi.org/10.1007/BFb0079068}.

\bibitem{souriau1997structure}
J.-M. Souriau, {\em Structure of dynamical systems: A symplectic view of
  physics}, vol.~149 of {\em Progress in Mathematics}.
\newblock Birkh\"auser, 1997.
\newblock Translated from the French by C. H. Cushman-de Vries, Translation
  edited and with a preface by R. H. Cushman and G. M. Tuynman.

\bibitem{woodhouse1992}
N.~M.~J. Woodhouse, {\em Geometric quantization}.
\newblock Oxford Mathematical Monographs. The Clarendon Press, Oxford
  University Press, New York, second~ed., 1992.

\bibitem{Kirillov_2001}
A.~A. Kirillov,
  \href{http://dx.doi.org/10.1007/978-3-662-06791-8_2}{``Geometric
  quantization,''} in {\em Dynamical systems, {IV}}, vol.~4 of {\em
  Encyclopaedia of Mathematical Sciences}, pp.~139--176.
\newblock Springer, Berlin, 2001.
\newblock \url{https://doi.org/10.1007/978-3-662-06791-8_2}.

\bibitem{BFFLS}
F.~Bayen, M.~Flato, C.~Fronsdal, A.~Lichnerowicz, and D.~Sternheimer, ``Quantum
  mechanics as a deformation of classical mechanics,''
  \href{http://dx.doi.org/10.1007/BF00399745}{{\em Lett. Math. Phys.}
  {\bfseries 1} no.~6, (1975/77) 521--530}.
  \url{https://doi.org/10.1007/BF00399745}.

\bibitem{BFFLS1}
F.~Bayen, M.~Flato, C.~Fronsdal, A.~Lichnerowicz, and D.~Sternheimer,
  ``Deformation theory and quantization. {I}. {D}eformations of symplectic
  structures,'' \href{http://dx.doi.org/10.1016/0003-4916(78)90224-5}{{\em Ann.
  Physics} {\bfseries 111} no.~1, (1978) 61--110}.
  \url{https://doi.org/10.1016/0003-4916(78)90224-5}.

\bibitem{BFFLS2}
F.~Bayen, M.~Flato, C.~Fronsdal, A.~Lichnerowicz, and D.~Sternheimer,
  ``Deformation theory and quantization. {II}. {P}hysical applications,''
  \href{http://dx.doi.org/10.1016/0003-4916(78)90225-7}{{\em Ann. Physics}
  {\bfseries 111} no.~1, (1978) 111--151}.
  \url{https://doi.org/10.1016/0003-4916(78)90225-7}.

\bibitem{dewildelecomte1983}
M.~De~Wilde and P.~B.~A. Lecomte, ``Existence of star-products and of formal
  deformations of the {P}oisson {L}ie algebra of arbitrary symplectic
  manifolds,'' \href{http://dx.doi.org/10.1007/BF00402248}{{\em Lett. Math.
  Phys.} {\bfseries 7} no.~6, (1983) 487--496}.
  \url{https://doi.org/10.1007/BF00402248}.

\bibitem{fedosovdeform}
B.~V. Fedosov, ``A simple geometrical construction of deformation
  quantization,'' {\em J. Differential Geom.} {\bfseries 40} no.~2, (1994)
  213--238. \url{http://projecteuclid.org/euclid.jdg/1214455536}.

\bibitem{kontsevich2003}
M.~Kontsevich, ``Deformation quantization of {P}oisson manifolds,''
  \href{http://dx.doi.org/10.1023/B:MATH.0000027508.00421.bf}{{\em Lett. Math.
  Phys.} {\bfseries 66} no.~3, (2003) 157--216}.
  \url{https://doi.org/10.1023/B:MATH.0000027508.00421.bf}.

\bibitem{Berezin_1975}
F.~A. Berezin, ``General concept of quantization,'' {\em Comm. Math. Phys.}
  {\bfseries 40} (1975) 153--174.
  \url{http://projecteuclid.org/euclid.cmp/1103860463}.

\bibitem{Corradini:2020vqa}
O.~Corradini, E.~Latini, and A.~Waldron, ``Quantum {D}arboux theorem,''
  \href{http://dx.doi.org/10.1103/physrevd.103.105021}{{\em Phys. Rev. D}
  {\bfseries 103} no.~10, (2021) Paper No. 105021, 19}.
  \url{https://doi.org/10.1103/physrevd.103.105021}.

\bibitem{He}
Z.~He, {\em Odd dimensional symplectic manifolds}.
\newblock ProQuest LLC, Ann Arbor, MI, 2010.
\newblock
  \url{http://gateway.proquest.com/openurl?url_ver=Z39.88-2004&rft_val_fmt=info:ofi/fmt:kev:mtx:dissertation&res_dat=xri:pqdiss&rft_dat=xri:pqdiss:0822892}.
\newblock Thesis (Ph.D.)--Massachusetts Institute of Technology.

\bibitem{Lin2013}
Y.~Lin, ``Lefschetz contact manifolds and odd dimensional symplectic
  geometry,'' \href{http://dx.doi.org/10.1016/j.optcom.2013.01.060}{{\em Optics
  Communications} {\bfseries 296} (June, 2013) 17–24}. arXiv:1311.1431
  [math].

\bibitem{Vaisman1983}
I.~Vaisman, ``Geometric quantization on presymplectic manifolds,''
  \href{http://dx.doi.org/10.1007/BF01471212}{{\em Monatsh. Math.} {\bfseries
  96} no.~4, (1983) 293--310}. \url{https://doi.org/10.1007/BF01471212}.

\bibitem{Gotay_Śniatycki_1981}
M.~J. Gotay and J.~e. \'Sniatycki, ``On the quantization of presymplectic
  dynamical systems via coisotropic imbeddings,'' {\em Comm. Math. Phys.}
  {\bfseries 82} no.~3, (1981/82) 377--389.
  \url{http://projecteuclid.org/euclid.cmp/1103920596}.

\bibitem{Günther_1980}
C.~G\"unther, ``Presymplectic manifolds and the quantization of relativistic
  particle systems,'' in {\em Differential geometrical methods in mathematical
  physics}, vol.~836 of {\em Lecture Notes in Mathematics}, pp.~383--400.
\newblock Springer, Berlin, 1980.

\bibitem{Herczeg:2017xxy}
G.~Herczeg and A.~Waldron, ``{Contact Geometry and Quantum Mechanics},''
  \href{http://dx.doi.org/10.1016/j.physletb.2018.04.008}{{\em Phys. Lett. B}
  {\bfseries 781} (2018) 312--315},
  \href{http://arxiv.org/abs/1709.04557}{{\ttfamily arXiv:1709.04557
  [hep-th]}}.

\bibitem{Herczeg:2018hup}
G.~Herczeg, E.~Latini, and A.~Waldron, ``Contact quantization: quantum
  mechanics = parallel transport,''
  \href{http://dx.doi.org/10.5817/am2018-5-281}{{\em Arch. Math. (Brno)}
  {\bfseries 54} no.~5, (2018) 281--298}.
  \url{https://doi.org/10.5817/am2018-5-281}.

\bibitem{alekseevsky}
D.~V. Alekseevskii, ``Contact homogeneous spaces,''
  \href{http://dx.doi.org/10.1007/BF01077337}{{\em Functional Analysis and Its
  Applications} {\bfseries 24} no.~4, (1990) 324--325}.
  \url{https://doi.org/10.1007/BF01077337}.

\bibitem{alekseevskysymmetric}
D.~V. Alekseevskii and C.~Gorodski, ``Semisimple symmetric contact spaces,''
  \href{http://dx.doi.org/10.1016/j.indag.2020.09.008}{{\em Indagationes
  Mathematicae} {\bfseries 31} no.~6, (2020) 1110--1133}.
  \url{https://doi.org/10.1016/j.indag.2020.09.008}.

\bibitem{Geiges2008}
H.~Geiges, \href{http://dx.doi.org/10.1017/CBO9780511611438}{{\em An
  introduction to contact topology}}, vol.~109 of {\em Cambridge Studies in
  Advanced Mathematics}.
\newblock Cambridge University Press, Cambridge, 2008.
\newblock \url{https://doi.org/10.1017/CBO9780511611438}.

\bibitem{LibermannMarle1987}
P.~Libermann and C.-M. Marle,
  \href{http://dx.doi.org/10.1007/978-94-009-3807-6}{{\em Symplectic geometry
  and analytical mechanics}}, vol.~35 of {\em Mathematics and Its
  Applications}.
\newblock D. Reidel Publishing Company, Dordrecht, 1987.
\newblock \url{https://doi.org/10.1007/978-94-009-3807-6}.
\newblock Translated from the French by Bertram Eugene Schwarzbach.

\bibitem{CiagliaCruzMarmo2018}
F.~M. Ciaglia, H.~Cruz, and G.~Marmo, ``Contact manifolds and dissipation,
  classical and quantum,''
  \href{http://dx.doi.org/10.1016/j.aop.2018.09.012}{{\em Ann. Physics}
  {\bfseries 398} (2018) 159--179}.
  \url{https://doi.org/10.1016/j.aop.2018.09.012}.

\bibitem{Bravetti2017}
A.~Bravetti, H.~Cruz, and D.~Tapias, ``Contact {H}amiltonian mechanics,''
  \href{http://dx.doi.org/10.1016/j.aop.2016.11.003}{{\em Ann. Physics}
  {\bfseries 376} (2017) 17--39}.
  \url{https://doi.org/10.1016/j.aop.2016.11.003}.

\bibitem{Bravetti2017v2}
A.~Bravetti, ``Contact {H}amiltonian dynamics: the concept and its use,''
  \href{http://dx.doi.org/10.3390/e19100535}{{\em Entropy} {\bfseries 19}
  no.~10, (2017) Paper No. 535, 12}. \url{https://doi.org/10.3390/e19100535}.

\bibitem{leonlainz2019}
M.~de~Le\'on and M.~Lainz~Valc\'azar, ``Contact {H}amiltonian systems,''
  \href{http://dx.doi.org/10.1063/1.5096475}{{\em J. Math. Phys.} {\bfseries
  60} no.~10, (2019) 102902, 18}. \url{https://doi.org/10.1063/1.5096475}.

\bibitem{leonlainz2021}
M.~d. León and M.~Lainz, ``A review on contact hamiltonian and lagrangian
  systems,''. \url{http://arxiv.org/abs/2011.05579}. arXiv:2011.05579
  [math-ph].

\bibitem{chwdynamics}
R.~Casals, G.~Herczeg, and A.~Waldron, ``Dynamical quantization of contact
  structures,'' \href{http://dx.doi.org/10.4310/atmp.2023.v27.n3.a7}{{\em Adv.
  Theor. Math. Phys.} {\bfseries 27} no.~3, (2023) 881--959}.
  \url{https://doi.org/10.4310/atmp.2023.v27.n3.a7}.

\bibitem{robinson1989metaplectic}
P.~L. Robinson and J.~H. Rawnsley, ``The metaplectic representation, {${\rm
  Mp}^c$} structures and geometric quantization,''
  \href{http://dx.doi.org/10.1090/memo/0410}{{\em Mem. Amer. Math. Soc.}
  {\bfseries 81} no.~410, (1989) iv+92}.
  \url{https://doi.org/10.1090/memo/0410}.

\bibitem{symplecticdirac}
K.~Habermann and L.~Habermann, \href{http://dx.doi.org/10.1007/b138212}{{\em
  Introduction to symplectic {D}irac operators}}, vol.~1887 of {\em Lecture
  Notes in Mathematics}.
\newblock Springer-Verlag, Berlin, 2006.
\newblock \url{https://doi.org/10.1007/b138212}.

\bibitem{symplecticdirac2}
M.~Cahen, S.~Gutt, and J.~Rawnsley, ``Symplectic {D}irac operators and
  {$Mp^{\rm c}$}-structures,''
  \href{http://dx.doi.org/10.1007/s10714-011-1239-x}{{\em Gen. Relativity
  Gravitation} {\bfseries 43} no.~12, (2011) 3593--3617}.
  \url{https://doi.org/10.1007/s10714-011-1239-x}.

\bibitem{Elfimov_2022}
B.~M. Elfimov and A.~A. Sharapov, ``Deformation quantization of contact
  manifolds,'' \href{http://dx.doi.org/10.1007/s11005-022-01621-3}{{\em Lett.
  Math. Phys.} {\bfseries 112} no.~6, (2022) Paper No. 124, 21}.
  \url{https://doi.org/10.1007/s11005-022-01621-3}.

\bibitem{rieffelheisenberg}
M.~A. Rieffel, ``Deformation quantization of {H}eisenberg manifolds,'' {\em
  Comm. Math. Phys.} {\bfseries 122} no.~4, (1989) 531--562.
  \url{http://projecteuclid.org/euclid.cmp/1104178588}.

\bibitem{rieffeloperator}
M.~A. Rieffel,
  \href{http://dx.doi.org/10.1090/pspum/051.1/1077400}{``Deformation
  quantization and operator algebras,''} in {\em Operator theory: operator
  algebras and applications, {P}art 1 ({D}urham, {NH}, 1988)}, vol.~51, Part 1
  of {\em Proc. Sympos. Pure Math.}, pp.~411--423.
\newblock Amer. Math. Soc., Providence, RI, 1990.
\newblock \url{https://doi.org/10.1090/pspum/051.1/1077400}.

\bibitem{rieffelquestions}
M.~A. Rieffel, \href{http://dx.doi.org/10.1090/conm/228/03294}{``Questions on
  quantization,''} in {\em Operator algebras and operator theory ({S}hanghai,
  1997)}, vol.~228 of {\em Contemp. Math.}, pp.~315--326.
\newblock Amer. Math. Soc., Providence, RI, 1998.
\newblock \url{https://doi.org/10.1090/conm/228/03294}.

\bibitem{rieffelconvolution}
M.~A. Rieffel, ``Lie group convolution algebras as deformation quantizations of
  linear {P}oisson structures,'' \href{http://dx.doi.org/10.2307/2374874}{{\em
  Amer. J. Math.} {\bfseries 112} no.~4, (1990) 657--685}.
  \url{https://doi.org/10.2307/2374874}.

\bibitem{waldmann2019}
S.~Waldmann, ``Convergence of star products: from examples to a general
  framework,'' \href{http://dx.doi.org/10.4171/emss/31}{{\em EMS Surv. Math.
  Sci.} {\bfseries 6} no.~1-2, (2019) 1--31}.
  \url{https://doi.org/10.4171/emss/31}.

\bibitem{landsmancoadjoint}
N.~P. Landsman, ``Strict quantization of coadjoint orbits,''
  \href{http://dx.doi.org/10.1063/1.532644}{{\em J. Math. Phys.} {\bfseries 39}
  no.~12, (1998) 6372--6383}. \url{https://doi.org/10.1063/1.532644}.

\bibitem{hawkins2008}
E.~Hawkins, ``An obstruction to quantization of the sphere,''
  \href{http://dx.doi.org/10.1007/s00220-008-0517-2}{{\em Comm. Math. Phys.}
  {\bfseries 283} no.~3, (2008) 675--699}.
  \url{https://doi.org/10.1007/s00220-008-0517-2}.

\bibitem{fedosovbook}
B.~V. Fedosov, {\em Deformation quantization and index theory}, vol.~9 of {\em
  Mathematical Topics}.
\newblock Akademie Verlag, Berlin, 1996.

\bibitem{holstein}
T.~{Holstein} and H.~{Primakoff}, ``{Field Dependence of the Intrinsic Domain
  Magnetization of a Ferromagnet},''
  \href{http://dx.doi.org/10.1103/PhysRev.58.1098}{{\em Physical Review}
  {\bfseries 58} no.~12, (Dec., 1940) 1098--1113}.

\bibitem{Okubo_1975}
S.~Okubo, ``Algebraic identities among {$U(n)$} infinitesimal generators,''
  \href{http://dx.doi.org/10.1063/1.522550}{{\em J. Mathematical Phys.}
  {\bfseries 16} (1975) 528--535}. \url{https://doi.org/10.1063/1.522550}.

\bibitem{Klein_Marshalek_1991}
A.~Klein and E.~R. Marshalek, ``Boson realizations of {L}ie algebras with
  applications to nuclear physics,''
  \href{http://dx.doi.org/10.1103/RevModPhys.63.375}{{\em Rev. Modern Phys.}
  {\bfseries 63} no.~2, (1991) 375--558}.
  \url{https://doi.org/10.1103/RevModPhys.63.375}.

\bibitem{Palev_1997}
T.~D. Palev, ``The {H}olstein-{P}rimakoff and {D}yson realizations for the
  {L}ie superalgebra {${\rm gl}(m/n+1)$},''
  \href{http://dx.doi.org/10.1088/0305-4470/30/23/023}{{\em J. Phys. A}
  {\bfseries 30} no.~23, (1997) 8273--8278}.
  \url{https://doi.org/10.1088/0305-4470/30/23/023}.

\bibitem{Oh_Rim_1997}
P.~Oh and C.~Rim, ``Holstein-{P}rimakoff realizations on coadjoint orbits,''
  \href{http://dx.doi.org/10.1142/S0217732397000169}{{\em Modern Phys. Lett. A}
  {\bfseries 12} no.~3, (1997) 163--172}.
  \url{https://doi.org/10.1142/S0217732397000169}.

\bibitem{knappbeyondliegroups}
A.~W. Knapp, {\em Lie groups beyond an introduction}, vol.~140 of {\em Progress
  in Mathematics}.
\newblock Birkh\"auser, Boston, second~ed., 2002.

\bibitem{montgomerysphere}
D.~Montgomery and H.~Samelson, ``Transformation groups of spheres,''
  \href{http://dx.doi.org/10.2307/1968975}{{\em Ann. of Math. (2)} {\bfseries
  44} (1943) 454--470}. \url{https://doi.org/10.2307/1968975}.

\bibitem{borelsphere}
A.~Borel, ``Some remarks about {L}ie groups transitive on spheres and tori,''
  \href{http://dx.doi.org/10.1090/S0002-9904-1949-09251-0}{{\em Bull. Amer.
  Math. Soc.} {\bfseries 55} (1949) 580--587}.
  \url{https://doi.org/10.1090/S0002-9904-1949-09251-0}.

\bibitem{besseeinstein}
A.~L. Besse, \href{http://dx.doi.org/10.1007/978-3-540-74311-8}{{\em Einstein
  manifolds}}, vol.~10 of {\em Classics in Mathematics}.
\newblock Springer, Berlin, Heidelberg, 1987.
\newblock \url{https://doi.org/10.1007/978-3-540-74311-8}.

\bibitem{sharpe}
R.~W. Sharpe, {\em Differential geometry: Cartan's generalization of Klein's
  Erlangen program}, vol.~166 of {\em Graduate Texts in Mathematics}.
\newblock Springer-Verlag, New York, 1997.
\newblock With a foreword by S. S. Chern.

\bibitem{kobayashinomizu2}
S.~Kobayashi and K.~Nomizu, {\em Foundations of differential geometry, Volume
  II}.
\newblock Wiley Classics Library. John Wiley \& Sons, Inc., New York, 1996.
\newblock Reprint of the 1969 original, A Wiley-Interscience Publication.

\bibitem{cap2009parabolic}
A.~\v{C}ap and J.~Slov\'ak, \href{http://dx.doi.org/10.1090/surv/154}{{\em
  Parabolic geometries {I}: Background and general theory}}, vol.~154 of {\em
  Mathematical Surveys and Monographs}.
\newblock American Mathematical Society, Providence, RI, 2009.
\newblock \url{https://doi.org/10.1090/surv/154}.

\bibitem{kobayashinomizu1}
S.~Kobayashi and K.~Nomizu, {\em Foundations of differential geometry, Volume
  I}.
\newblock Wiley Classics Library. John Wiley \& Sons, Inc., New York, 1996.
\newblock Reprint of the 1963 original, A Wiley-Interscience Publication.

\bibitem{inonuwigner}
E.~Inonu and E.~P. Wigner, ``On the contraction of groups and their
  representations,'' \href{http://dx.doi.org/10.1073/pnas.39.6.510}{{\em Proc.
  Nat. Acad. Sci. U.S.A.} {\bfseries 39} (1953) 510--524}.
  \url{https://doi.org/10.1073/pnas.39.6.510}.

\bibitem{Groenewold}
H.~J. Groenewold, ``On the principles of elementary quantum mechanics,'' {\em
  Physica} {\bfseries 12} (1946) 405--460.

\bibitem{gotays2}
M.~J. Gotay, H.~Grundling, and C.~A. Hurst, ``A {G}roenewold-van {H}ove theorem
  for {$S^2$},'' \href{http://dx.doi.org/10.1090/S0002-9947-96-01559-0}{{\em
  Trans. Amer. Math. Soc.} {\bfseries 348} no.~4, (1996) 1579--1597}.
  \url{https://doi.org/10.1090/S0002-9947-96-01559-0}.

\bibitem{schmüdgen1990}
K.~Schm\"udgen, \href{http://dx.doi.org/10.1007/978-3-0348-7469-4}{{\em
  Unbounded operator algebras and representation theory}}, vol.~37 of {\em
  Operator Theory: Advances and Applications}.
\newblock Birkh\"auser Verlag, Basel, 1990.
\newblock \url{https://doi.org/10.1007/978-3-0348-7469-4}.

\bibitem{schmüdgen2020}
K.~Schm\"udgen, \href{http://dx.doi.org/10.1007/978-3-030-46366-3}{{\em An
  invitation to unbounded representations of {$*$}-algebras on {H}ilbert
  space}}, vol.~285 of {\em Graduate Texts in Mathematics}.
\newblock Springer, Cham, 2020.
\newblock \url{https://doi.org/10.1007/978-3-030-46366-3}.

\bibitem{forgerhess1979}
M.~Forger and H.~Hess, ``Universal metaplectic structures and geometric
  quantization,'' {\em Comm. Math. Phys.} {\bfseries 64} no.~3, (1979)
  269--278. \url{http://projecteuclid.org/euclid.cmp/1103904723}.

\bibitem{hess1981}
H.~Hess, \href{http://dx.doi.org/10.1007/3-540-10578-6_19}{``On a geometric
  quantization scheme generalizing those of {K}ostant-{S}ouriau and
  {C}zy\.z,''} in {\em Differential geometric methods in mathematical physics},
  vol.~139 of {\em Lecture Notes in Physics}, pp.~1--35.
\newblock Springer, Berlin, Heidelberg, 1981.
\newblock \url{https://doi.org/10.1007/3-540-10578-6_19}.

\end{thebibliography}\endgroup
\end{document}